\documentclass[12pt,a4paper]{amsart}
\usepackage[margin=2.5cm]{geometry}
\usepackage{tikz,tkz-graph}
\usepackage{float}
\usetikzlibrary{arrows}
\usetikzlibrary{arrows.meta}
\usetikzlibrary{calc}
\usepackage{pgfplots}
\pgfplotsset{every axis/.append style={
          axis x line=middle,  
          axis y line=middle,  
          axis line style={-,color=blue}, 
          xlabel={$x$},     
          ylabel={$y$},     
      }}
\pgfplotsset{compat=1.13}

\usepackage[latin1]{inputenc}
\usepackage{amsmath}
\usepackage{amsthm}
\usepackage{amsfonts}
\usepackage{amssymb}
\usepackage{mathtools}
\usepackage{enumerate}
\usepackage{graphicx}
\usepackage{hyperref}
\usepackage{nicefrac}
\usepackage{enumitem}
\usepackage[all,cmtip]{xy}

\usepackage{tikz}
\usetikzlibrary{arrows, intersections, calc, matrix}

\DeclareMathOperator{\Hom}{Hom}

\DeclareMathOperator{\ord}{ord}

\DeclareMathOperator{\Supp}{Supp}

\def\rel{reduced }
\def\RHS{$\mathbb QH S^3$ }

\def\cS{{\mathcal S}}
\def\V{{V}}
\def\I{{I}}
\def\J{{J}}
\def\K{{K}}

\def\ZZ{\mathbb{Z}}

\def\NN{\mathbb{N}}
\def\CC{\mathbb{C}}

\def\Q{\mathbb{Q}}
\def\QQ{\mathbb{Q}}

\def\RR{\mathbb{R}}

\def\cE{\mathcal{E}}

\def\cO{\mathcal{O}}

\def\cF{\mathcal{F}}
\def\cK{\mathcal{K}}

\def\R{\Delta}

\def\div{\textrm{div}}
\def\pc{\textrm{pc}}
\def\mpc{\textrm{mpc}}

\def\tt{\mathbf{t}}
\def\ZK{Z_K}
\def\qp#1{\widetilde{#1}}

\def\eqmap#1{\smash{\mathop{=}\limits^{#1}}}

\newtheorem{thm}{Theorem}[section] 
\newtheorem{prop}{Proposition}[section]
\newtheorem{question}{Question}

\newtheorem{cor}{Corollary}[section]
\newtheorem{lemma}{Lemma}[section]
\newtheorem{def-lemma}{Definition-Lemma}[section]

\theoremstyle{remark}
\newtheorem{rem}{Remark}[section]

\theoremstyle{definition}
\newtheorem{dfn}{Definition}[section]
\newtheorem{exam}{Example}[section]

\makeatletter
\let\c@lemma\c@thm
\let\c@prop\c@thm
\let\c@propdef\c@thm
\let\c@proper\c@thm
\let\c@problem\c@thm
\let\c@conj\c@thm
\let\c@cor\c@thm
\let\c@rem\c@thm
\let\c@dfn\c@thm
\let\c@notation\c@thm
\let\c@exam\c@thm
\makeatother

\title
[The delta invariant... II:
Poincar\'e series and topological aspects]
{The delta invariant of curves on rational surfaces II:
Poincar\'e series and topological aspects}

\author[J.I. Cogolludo]{Jos{\'e} Ignacio Cogolludo-Agust{\'i}n}
\address{Departamento de Matem\'aticas, IUMA\\
Universidad de Zaragoza\\
C.~Pedro Cerbuna 12\\
50009 Zaragoza, Spain}
\email{jicogo@unizar.es}

\author[T. L\'aszl\'o]{Tam\'as L\'aszl\'o}
\address{Old: Alfr\'ed R\'enyi Institute of Mathematics,
Hungarian Academy of Sciences,
Re\'altanoda utca 13-15, H-1053, Budapest, Hungary \newline \hspace*{4mm}
New: Babe\c{s}-Bolyai University, Str. Mihail Kog\u{a}lniceanu nr. 1, 400084 Cluj-Napoca, Romania}
\email{laszlo.tamas@renyi.mta.hu}

\author[J. Mart\'in]{Jorge Mart\'in-Morales}
\address{Centro Universitario de la Defensa-IUMA\\
Academia General Militar\\
Ctra.~de Huesca s/n.\\
50090, Zaragoza, Spain}
\email{jorge@unizar.es}

\author[A. N\'emethi]{Andr\'as N\'emethi}
\address{Alfr\'ed R\'enyi Institute of Mathematics,
Hungarian Academy of Sciences,
Re\'altanoda utca 13-15, H-1053, Budapest, Hungary \newline
 \hspace*{4mm} ELTE - University of Budapest, Dept. of Geometry, Budapest, Hungary \newline \hspace*{4mm}
BCAM - Basque Center for Applied Math.,
Mazarredo, 14 E48009 Bilbao, Basque Country, Spain}
\email{nemethi.andras@renyi.mta.hu }

\thanks{The first and third authors are partially supported by MTM2016-76868-C2-2-P and
Gobierno de Arag{\'o}n (Grupo de referencia ``{\'A}lgebra y Geometr{\'i}a'')
cofunded by Feder 2014-2020 ``Construyendo Europa desde Arag\'on''.
The third author is also partially supported by FQM-333 ``Junta de Andaluc{\'\i}a''. \\
The second and fourth authors are supported by NKFIH Grant ``\'Elvonal (Frontier)'' KKP 126683. 
The second author was also supported by ERCEA Consolidator Grant 615655 - NMST, 
and partially by the Basque Government through the BERC 2018-2021 program and Gobierno Vasco Grant IT1094-16, 
by the Spanish Ministry of Science, Innovation and Universities: BCAM Severo Ochoa accreditation SEV-2017-0718.}

\begin{document}

\begin{abstract}
In this article we study abstract and embedded invariants of reduced curve germs
via topological techniques. One of the most important numerical analytic invariants of an abstract curve is 
its delta invariant. Our primary goal is to develop delta invariant formulae for curves embedded in rational 
singularities in terms of embedded data. The topological machinery not only produces formulae, but it also 
creates deep connections with the theory of (analytical and topological) multivariable Poincar\'e series.
\end{abstract}

\maketitle

\tableofcontents

\section{Introduction}
In this article we study crucial abstract and embedded invariants of reduced curve germs
via topological techniques. One of the most important numerical analytic invariants of an abstract
curve is its delta invariant. Our primary goal is to develop delta invariant formulae for curves
embedded in rational singularities in terms of embedded data. Some of the formulae, which will be
considered, were already found in \cite{kappa} using algebro-geometric methods. Nevertheless, in the present note
we run a completely different (mostly topological) machinery, which produces some additional formulae and
also creates deep connections with the theory of (analytical and topological) multivariable Poincar\'e series.

The motivation for this work comes from the common territory between the local and global study of a (reduced)
curve on a projective normal surface. In order to discuss this aspect both from local and global points of view,
in the sequel we make the following preparations.

\subsection{}\label{ss:intro1}
Let $(X,0)$ be a normal surface singularity and $(C,0)\subset (X,0)$ a reduced curve germ on it. Regarding the
normal surface singularity, for simplicity we will assume that its link $\Sigma$ is a rational homology sphere,
denoted by \RHS\ (e.g. rational singularities satisfy this restriction).

Then, we fix a good embedded resolution $\pi:\tilde X\to X$ of the pair $C\subset X$ and
consider the usual combinatorial package of the resolution (for details see section~\ref{sec:lattice}):
$E=\pi^{-1}(0)$ is the exceptional curve with its decomposition $\cup_vE_v$ into irreducible components,
$L=H_2(\tilde X, \ZZ)=\ZZ\langle E_v\rangle_v$ is the lattice of $\pi$ endowed with the negative definite
intersection form $(E_v,E_w)_{v,w}$. We identify the dual lattice $L'$ with those
rational cycles $\ell'\in L\otimes \QQ$ for which $(\ell',\ell)\in \ZZ$ for any $\ell\in L$.
Then $L'/L$ is the finite group $H_1(\partial \tilde X,\ZZ)$ ($\partial \tilde X=\Sigma$), which will be denoted by $H$.
We set $[\ell']$ for the class of $\ell'\in L'$ in $H$.

Let $K_\pi\in L'$ be the {\it canonical cycle} of $\pi$ (cf.~\eqref{eq:KX}), and we consider for any $\ell'\in L'$ the
Riemann--Roch expression $\chi(\ell'):=-(\ell', \ell'+K_\pi)/2$ (note that if $\ell\in L$ is effective and
non-zero then $\chi(\ell)=h^0(\cO_\ell)-h^1(\cO_\ell)$). Throughout this article, for a fixed good resolution $\pi$ the
notation $Z_K:=-K_{\pi}$ will be used.

Let $\cS'$ be the Lipman (anti-nef) cone $\{\ell'\in L'\,:\, (\ell', E_v)\leq 0\ \mbox{for all $v$}\}$.
By the negative definiteness of $(-,-)$ we know that $\cS'$ sits in the first quadrant of $L\otimes {\mathbb R}$.
One shows that for any $h\in H$ there exists a unique $s_h\in \cS'$ such that $[s_h]=h$ and $s_h$ is minimal with these
two properties. The cycle $s_h$ is zero only if $h=0$, and it usually is rather arithmetical and hard to find explicitly,
cf.~\cite{NemOSZ}.

Regarding the embedded curve $C$, let $\pi^*(C)$ be its {\it total transform}. It decomposes into $\ell_C'+\widetilde{C}$,
where $\widetilde{C}\subset \tilde X$ is the strict transform and $\ell'_C\in L'$ is a rational cycle determined uniquely
by the property that $\ell_C'+\widetilde{C}$ is a numerically trivial divisor.
The embedded topological type of $C$ is basically determined by the information on how many components of $\widetilde{C}$
intersect each $E_v$.

Then, as a first step we consider two numerical invariants: \textit{the delta invariant $\delta(C)$ of the abstract curve
germ $(C,0)$ and the embedded topological invariant $\chi(-\ell'_C)$ of the pair $C\subset X$.}

\subsection{Motivation (I)}
From a global analytic/algebraic point of view, there are several articles studying the generalized Riemann--Roch theorem for
Weil divisors on projective normal surfaces. In this context, the \textit{usual} formula valid for smooth surfaces is modified by
a correction term associated with the local singular points. E.g., Blache in~\cite{Blache-RiemannRoch} provides the following
formula, valid for any projective normal surface $Y$ and any Weil divisor $D$ of $Y$
$$\chi(\cO_Y(D))=\chi(\cO_Y)+ \frac{1}{2} ( D, D-K_Y)-\sum_{y\in {\rm Sing}(Y)} A_{Y,y}(-D),$$
where the local correction contribution associated with the singular points is encoded by a map
$A_{Y,y}:{\rm Weil}(Y,y)/{\rm Cartier}(Y,y) \to \QQ$, which additionally satisfies
$$A_{Y,y}(D)= \chi(-\ell_D')-\delta(D)$$ whenever $(D,y)\subset (Y,y)$ is a reduced curve germ.
In particular, if $(D,y)$ is Cartier then $\delta(D)=\chi(-\ell_D')$.
This expression identifies an abstract analytic invariant with an embedded topological invariant.
(Different aspects of this identity are discussed in~\cite[\S3.4]{kappa}.)

Then, one of our guiding question is the following:
\vspace{2mm}

\begin{question}\label{question:1}
In general, for a given $(X,0)$, can $A_{X,0}$, or equivalently $\delta(C)$, be read from the
local embedded topological type of the pair $(C,0)\subset (X,0)$?
\end{question}

\subsection{Motivation (II)}
Regarding the curve and normal surface germs, in a series of articles
--\,\cite{CDG-Poincare,CDGZ-PScurves,CDGZ-equivariantPS,CDGZ-HilbertFunction,CDGZ-equivariantPSZeta,CDGZ-PSfiltrations,
CDGZ-equivariantPSTopology,CDGZ-topologicaltype}\,--
A.~Campillo, F.~Delgado, and S.~Gusein-Zade introduced the \textit{multivariable analytical Poincar\'e series}
associated with different filtrations of the corresponding local ring of functions. In \cite{cdg} they proved
that for a plane curve singularity germ the Poincar\'e series is equivalent with the Alexander polynomial of its link.
Moreover, for a normal surface singularity $(X,0)$ the analytical Poincar\'e series --- as well as its topological analogue,
the topological Poincar\'e series\,--- codify a vast amount of information about $(X,0)$. For example, in the case of
$(X,0)$ with \RHS link, by a regularization procedure (measuring the asymptotic behavior of the coefficients of the series)
the geometric genus of $(X,0)$ can be interpreted as the \textit{periodic constant} of the analytical series. Similarly, from
a topological point of view, the Seiberg--Witten invariant of the link of $(X,0)$ can be determined as the periodic constant
of the topological series (see section~\ref{sec:pg}).

On the one hand, these results highlight the importance of the Poincar\'e series. On the other hand, they immediately pose the
following questions in our context as well:
\vspace{2mm}

\begin{question}\label{question:2}
Can the delta invariant of an (abstract) curve germ $(C,0)$ be obtained from its Poincar\'e series?
\end{question}

\begin{question}\label{question:3}
If we consider $(C,0)\subset (X,0)$, can the delta invariant be decoded from the Poincar\'e series of $(X,0)$?
Or, more generally, what kind of information of the embedded curve can be derived from the Poincar\'e series of the surface?
\end{question}

In this note we will produce numerical invariants from the series via their periodic constants (this is a
\textit{regularization} procedure, see \ref{ss:pc}).

\subsection{Main Results}
The main results of the present article can be divided into three parts. At a first glance, the first two parts are independent
from each other. The first part provides the $\delta$-invariant of an abstract curve as the periodic constant of an
analytic Poincar\'e series (associated with the abstract curve). The second one determines the periodic constant of the
topological Poincar\'e series associated with an embedded situation in terms of a finite sum. However, the third part
connects them and it creates the topological machinery, which produces concrete $\delta$-invariant formulae (in terms of the
embedded topological type) whenever the ambient space $(X,0)$ is rational.
In the sequel, we will introduce some notation in order to state the results of the corresponding parts.
\vspace{4mm}

{\bf Part I}. \ The first result of the present paper answers Question~\ref{question:2}.
Let $(C,0)$ be an (abstract) reduced curve germ and write
$C=\cup_{i\in I}C_i$ as the union of its irreducible components.
Denote by $P_{C_{J}}(\tt_J)$ the Poincar\'e series in variables $\{t_j\}_{j\in J}$
associated with $C_J:=\cup_{j\in J} C_j$ for a given $\emptyset\neq J\subset I$ (cf. section~\ref{sec:relabsP}). By~\cite{CDK} (cf. Lemma~\ref{lem:Ppoly}) $P_{C_J}$ is a polynomial whenever $|J|>1$, hence it makes sense to consider the evaluation $P_{C_{J}}(1_{J}):=P_{C_{J}}(\tt_{J})|_{t_j=1, \forall j\in J}$. Moreover, for a branch $C_i$ let ${\rm pc}(P_{C_i})$ be the periodic constant of $P_{C_i}$ as defined in section~\ref{ss:pc}.
Then in Theorem~\ref{thm:delta-sec} we obtain a formulae for the $\delta$-invariant of a possibly reducible curve germ in terms of the Poincar\'e series
as the following alternating sum:
\begin{align}
\label{intro:delta1}& \delta(C_i)= -{\rm pc}(P_{C_i}(t_i));\\
\label{intro:delta2}& \delta(C)=
\sum_{i\in I}\delta(C_i) +
\sum_{J\subset \I, \ |J|>1} (-1)^{|J|} \ P_{C_{J}}(1_{J}).
\end{align}
Though the theory of curves and their delta invariants is a classical subject,
to the best of our knowledge, formula~\eqref{intro:delta2} is new in the literature.

In particular, as it will be shown in section~\ref{ss:spec}, for plane curve singularities
formula~\eqref{intro:delta2} is equivalent to the well-known formula
$\delta(C)=\sum_i \delta(C_i)+\sum_{i<j}(C_i,C_j)$, since $P_{C_{i,j}}(1,1)=(C_i,C_j)$ gives the
intersection multiplicity and $P_{C_{J}}(1_{J})=0$ for any $|J|>3$.
\vspace{4mm}

{\bf Part II}. \
Following with the notation introduced in section~\ref{ss:intro1}, let us fix a normal surface
singularity $(X,0)$, a good resolution $\pi$ and let $V$ be the set of vertices of the dual
resolution graph $\Gamma$. Denote $r_h=\sum_{v\in V} l'_vE_v\in L'$ the unique representative of
$h$ (that is, $h=[r_h]\in L'/L=H$) with all $l'_v\in \Q\cap [0,1)$.
For $\ell'=\sum_v l'_vE_v$ and $\ell''=\sum_v l''_vE_v$ we write $\ell'\geq \ell'' $ if
$l'_v\geq l''_v $ for all $v\in V$. Let $Z(\tt)=\sum_{\ell'\in L'} z(\ell')\tt^{\ell'}$ be the topological Poincar\'e series, where $\tt^{\ell'}:=\prod_v t_v^{l'_v}$ for any $\ell'=\sum_vl'_vE_v$. The series $Z(\tt)$ decomposes as $\sum_{h\in H} Z_{h}(\tt)$, where
$Z_h(\tt):=\sum _{[\ell']=h}z(\ell')\tt^{\ell'}$. Furthermore, for any subset $I \subset V$
we write $Z_h(\tt_I):=Z_h(\tt)|_{t_i=1, \, i\not\in I}$. Fix $h\in H$ and $I\subset V$ and
define the {\it counting function of the coefficients of $Z(\tt_I)$} by
\begin{equation*}
Q_{h,I}: \{x\in L'\,:\, [x]=h\}\to \ZZ, \ \ \ \
Q_{h,I}(x):=\sum_{\ell'|_I \ngeq x |_I,
\, [\ell']=[x]} z(\ell').
\end{equation*}
Note that the sum above is finite, see section~\ref{ss:PZ}. Then, in~\cite{LNNdual} J.~Nagy
and the second and fourth authors proved that the periodic constant of the series $Z_h(\tt_I)$
(defined via the regularization procedure from section~\ref{ss:pc}) for any $I\subset V, I\neq \emptyset$ and $h\in H$ 
satisfies $\mathrm{pc}(Z_h(\tt_{I}))= Q_{[Z_K]-h,I}(Z_K-r_h)$. This result transforms the Ehrhart-MacDonald-Stanley 
reciprocity law from the theory of lattice polytopes to the level of series, giving the duality between the periodic 
constant of $Z_h(\tt_{I})$ and a finite sum of coefficients of the \textit{dual} series~$Z_{[Z_K]-h}(\tt_{I})$.

In this paper, we generalize the above duality for the twisted series $Z_{\ell_0'}(\tt):=\tt^{\ell_0'}\cdot Z(\tt)$ 
for some fixed $\ell_0'\in \cS'$ with $[\ell_0']=h_0$ and give the following {\it twisted duality} in Theorem~\ref{tdual} as follows:
\begin{equation}
\label{eq:pc1}
\mathrm{pc}((Z_{\ell'_0})_h(\tt_{\I}))=Q_{[\ZK]-h+h_0,\I}(\ZK-r_h+\ell'_0).
\end{equation}
In particular, for $h=0$ one obtains
\begin{equation}
\label{eq:pc2}
\mathrm{pc}((Z_{\ell'_0})_0(\tt_{\I}))= Q_{[\ZK]+h_0,\I}(\ZK+\ell'_0).
\end{equation}

It is worth emphasizing here that identity~\eqref{eq:pc1} has two important specializations:
the first one gives back the result from~\cite{LNNdual}, the other one is the case $h=0$ as stated in~\eqref{eq:pc2}, 
which will be applied to an embedded curve germ $(C,0)\subset (X,0)$.
\vspace{4mm}

{\bf Part III}. \ We answer here Questions~\ref{question:1} and~\ref{question:3} in the rational case.
We consider the pair $(C,0)\subset (X,0)$ and assume that $(X,0)$ is rational. Fix a good embedded
resolution of the pair $C\subset X$ and denote by $I_C\subset V$ the set of irreducible exceptional
divisors which intersect the strict transform of $(C,0)$. Moreover, we choose the resolution in such
a way that every component from $I_C$ intersects only one component of the strict transform.

First, the connection between Part I and II is derived in the rational case from the identity of
Campillo, Delgado, and Gusein-Zade~\cite{CDGZ-PScurves} (called the relative CDGZ-identity,
cf. section~\ref{ss:rCDGZ})
$$P_C(\tt_{I_C})=Z_0^C(\tt_{I_C}),$$
which identifies the Poincar\'e series of the curve germ $(C,0)$ with the reduction of the $h=0$-part
of the \textit{relative topological Poincar\'e series}
$Z^C(\tt)=Z(\tt)\cdot \Pi_{v\in I_C}(1-\tt^{E_v^*})$.

Second, motivated by the theory of adjoints and its application to singularity
theory~\cite{Libgober-alexander,Libgober-characteristic,ji-tesis} another embedded analytic invariant --\,the $\kappa$-invariant $\kappa_X(C)$\,-- appears in this picture. Its definition has been already given in~\cite{kappa} as a conceptual generalization of the special cases considered previously in~\cite{ji-tesis,ji-JJ-numerical,CM}. In general, $\kappa_X(C)$ is
defined as the counting function of the analytical Poincar\'e series of $(X,0)$ evaluated at
$Z_K+\ell'_C$. Thus, when $(X,0)$ is rational, by the CDGZ-identity it is expressed by the counting function of $Z(\tt)$. More precisely, in Theorem~\ref{thm:eqratsing} we prove the following result for the $\kappa$-invariant of
 a reduced curve germ $(C,0)$ embedded into a rational surface singularity $(X,0)$ as follows:
\begin{equation}
\label{eq:intro:kappa}
\kappa_X(C)=Q_{[\ZK+\ell'_C]}(\ZK+\ell'_C)=Q_{[\ZK+\ell'_C],\I_C}(\ZK+\ell'_C).
\end{equation}

Finally, the twisted duality allows us to connect the periodic constant of $P_C$ with $\kappa_X(C)$,
and to use different methods regarding the counting function of $Z(\tt)$ (and developed recently
by~\cite{LNehrhart,LSzPoincare,LNN,LNNdual}) in order to find the following explicit formula for the $\delta$-invariant
of a reduced curve germ $(C,0)$ embedded into a rational surface singularity $(X,0)$ as follows:
\begin{align}
\label{intro:delta3}& \delta(C)=\kappa_X(C)=\chi(Z_K+\ell'_C)-\chi(s_{[Z_K+\ell'_C]});\\
\label{intro:delta4}& A_{X,0}(C)=\chi(s_{[Z_K+\ell'_C]})=\chi(s_{[-\ell'_C]}).
\end{align}

Note that~\eqref{intro:delta3} and \eqref{intro:delta4} are also proven in~\cite{kappa} using completely analytical methods
(without any connection with Poincar\'e series). The above topological approach can be summarized in
the following schematic picture:
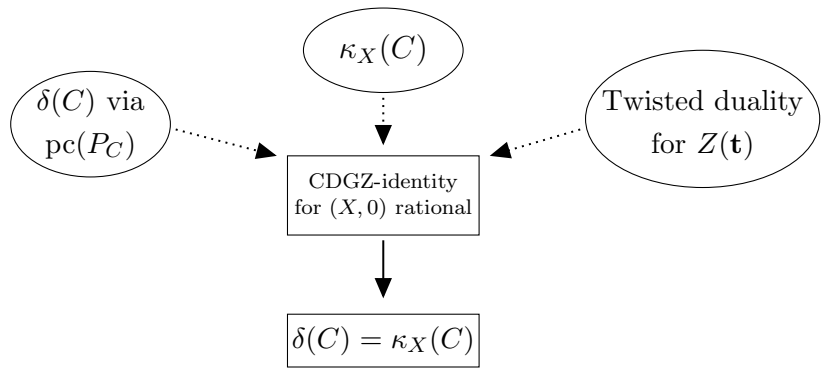
\vspace{1mm}

\begin{center}
\begin{tikzpicture}[scale=0.7]
\draw (-6.5,3.6) ellipse (1.5 and 1);
\draw (-1,5) ellipse (1.5 and 0.8); 
\draw (5,3.7) ellipse (2.2 and 1.3); 
\draw (-2.8,3) rectangle (0.8,1.5);
\draw (-2.8,0) rectangle (0.8,-1);
\node at (-6.5,4) {\small $\delta(C)$ via};
\node at (-6.5,3.2) {\small $\mathrm{pc}(P_C)$};
\node at (-1,5) {$\kappa_X(C)$};
\node at (5,4) {\small Twisted duality};
\node at (5,3.2) {\small for $Z(\mathbf{t})$};
\node at (-1,2.5) {\tiny CDGZ-identity};
\node at (-1,2) {\tiny for $(X,0)$ rational};
\node at (-1,-0.5) {\small $\delta(C)=\kappa_X(C)$};
\draw [-triangle 45,dotted,line width=0.75] (-4.9,3.5) -- (-3,3);
\draw [-triangle 45,dotted,line width=0.75] (-1,4.1) -- (-1,3.2);
\draw [-triangle 45,dotted,line width=0.75] (2.7,3.5) -- (1,3);
\draw [-triangle 45,line width=0.75] (-1,1.4) -- (-1,0.3);
\node at (-1,1.5) {};
\node at (-1,0.5) {};
\end{tikzpicture}
\end{center}

\subsection{Organization} The paper is organized as follows. Section~\ref{sec:prelim} contains the 
necessary ingredients of the paper such as the basic tools for study \RHS surface singularities, 
the theory of counting functions and periodic constants of multivariable series and useful results 
about Poincar\'e series of curves and surface singularities. Section~\ref{sec:kappa-inv}
reviews first the general definition of the $\kappa$-invariant from \cite{kappa}, then discusses its 
topological and reduced analogues together with their identification in the case of rational surface 
singularities. Then, section~\ref{sec:formula-delta} contains the first part of the results about the 
$\delta$-invariant of an abstract curve germ, section~\ref{sec:twisted-duality} is about the twisted 
duality of the topological Poincar\'e series and section~\ref{sec:deltakappa} presents the topological 
approach of the applications discussed above. We also give some examples in section~\ref{sec:examples}
illustrating the concepts and results of the paper.

\subsection*{Acknowledgments}
The first and third authors want to thank the Fulbright Program (within the Jos\'e Castillejo and Salvador de Madariaga
grants by Ministerio de Educaci\'on, Cultura y Deporte) for their financial support while writing this paper. They also want
to thank the University of Illinois at Chicago, especially Anatoly Libgober, Lawrence Ein, and Kevin Tucker for their warm
welcome and support in hosting them as well as their useful discussions.

The second author expresses his gratitude to the community of the R\'enyi Institute of Mathematics for their 
kindness and for the support he received from them during the time he worked at this renown institution.

\section{Preliminaries}
\label{sec:prelim}
\subsection{Lattices for surface singularities with \texorpdfstring{\RHS}{RHS} links}
\label{sec:lattice}
Let us consider a complex normal surface singularity $X=(X,0)$. Let $\pi:\tilde X\to X$ be a good resolution with dual
resolution graph $\Gamma$ whose set of vertices are denoted by $\V$. Let $\{E_v\}_{v\in \V}$ be the irreducible components
of the exceptional set $\pi^{-1}(0)$. We assume that the link $\Sigma$ is a rational homology sphere, i.e. $\Gamma$ is a 
connected tree and all $E_v$ are rational.

Define the lattice $L$ as $H_2(\tilde X,\ZZ)$, it is generated by the exceptional divisors $E_v$, $v\in \V$, that is,
$L=\oplus_{v\in \V} \ZZ\langle E_v \rangle$. In the homology exact sequence of the pair $(\tilde X, \Sigma)$ 
($\partial \tilde X=\Sigma$) one has $H_2(\Sigma,\ZZ)=0$, $H_1(\tilde X, \ZZ)=0$, hence the exact sequence has the form:
\begin{equation}
\label{eq:ses}
0 \to L \to H_2(\tilde X,\Sigma,\ZZ) \to H_1(\Sigma,\ZZ) \to 0.
\end{equation}
Set $L':= \Hom(H_2(\tilde X,\ZZ),\ZZ)$.
The Lefschetz-Poincar\'e duality $H_2(\tilde X,\Sigma,\ZZ)\cong H^2(\tilde X,\ZZ)$
defines a perfect pairing
$L\otimes H_2(\tilde X,\Sigma,\ZZ)\to \ZZ$. Hence $L'$
can be identified with $H_2(\tilde{X}, \Sigma, \ZZ)$. By~\eqref{eq:ses}
$L'/L\cong H_1(\Sigma,\ZZ)$, which will be denoted by $H$. (Note that even if $\Sigma$ is not
${\mathbb Q}HS^3$, since the intersection form on
$L$ is non--degenerate, $H_2(\tilde{X},\ZZ)\to H_2(\tilde{X},\Sigma, \ZZ)$ is injective, and
$L'/L= {\rm Tors}(H_1(\Sigma,\ZZ))$.)
Since the intersection form is non--degenerate, $L'$ embeds into $ L_{{\mathbb Q}}:=
 L\otimes {\mathbb Q}$,
and it can be identified with the rational cycles $\{\ell'\in L_{{\mathbb Q}}\,:\, (\ell',L)_{{\mathbb Q}}\in \ZZ\}$, where
$(\,,\,)$ denotes the intersection form on $L$ and $(\,,\,)_{{\mathbb Q}}$ its extension to $L_{{\mathbb Q}}$.
 Hence, in the sequel we regard $L'$ as $\oplus_{v\in \V} \ZZ\langle E^*_v \rangle$,
 the lattice generated by the rational cycles $E^*_v\in L_{{\QQ}}$,
$v\in \V$, where $(E_u^*,E_v)_{{\QQ}}=-\delta_{u,v}$ (Kronecker delta) for any $u,v\in \V$.

The elements $E^*_v$ have the following
geometrical interpretation as well: consider $\gamma_v\subset \tilde X$ a curvette associated with $E_v$, that is, a smooth
irreducible curve in $\tilde X$ intersecting $E_v$ transversally. Then
$\pi^*\pi_*(\gamma_v)=\gamma_v+E^*_v$.

Let $K_{\tilde X}$ be the canonical divisor of the smooth surface $\tilde X$. The canonical divisor in $(X,0)$ is defined as
$K_X:=\pi_*(K_{\tilde X})$. Note that $K_\pi:=K_{\tilde X}-\pi^*(K_X)$ has support on the exceptional set $\pi^{-1}(0)$.
The divisor $K_\pi$ is called the relative canonical divisor of $\pi$, and it is determined topologically by the
linear system of {\it adjunction relations}
\begin{equation}
\label{eq:KX}
(K_\pi+E_v,E_v)+2=0, \textrm{ for all } v\in \V.
\end{equation}
In some cases, it is more convenient to use the canonical cycle $\ZK:=-K_\pi$. Using~\eqref{eq:KX}, $Z_K$ can be written as
\begin{equation}\label{eq:ZK}
\ZK=E-\sum_{v\in \V} (2-\text{val}(v)) E^*_v,
\end{equation}
where $E=\sum_{v\in \V}E_v$ and $\text{val}(v)$ denotes the valence of $v$ in $\Gamma$.
In particular, $Z_K\in L'$.

For any non--zero effective cycle $\ell\in L$, from the cohomology exact sequence associated with
$0\to \cO_{\tilde X}(-\ell)\to \cO_{\tilde X}\to \cO_{\ell}\to 0$
one has $\chi(\cO_\ell)=\chi(\cO_{\tilde X})-\chi(\cO_{\tilde X}(-\ell))$.
 This, by the Riemann-Roch-Hirzebruch theorem gives
$\chi(\cO_\ell)=-(\ell, \ell-Z_K)/2$. This motivates to define $\chi(\ell'):=-(\ell', \ell'-Z_K)/2$
for any $\ell'\in L'$.

\subsection{\texorpdfstring{$H$}{H}-representatives and the Lipman cone}\label{ss:HrepLip}
For $\ell'_1,\ell'_2\in L_\QQ$ with $\ell'_i=\sum_v l'_{iv}E_v$ ($i=\{1,2\}$)
one considers an order relation $\ell'_1\geq \ell'_2$ defined coordinatewise by $l'_{1v}\geq l'_{2v}$
for all $v\in\V$. In particular,
$\ell'$ is an effective rational cycle if $\ell'\geq 0$.
We set also $\min\{\ell'_1,\ell'_2\}:= \sum_v\min\{l'_{1v},l'_{2v}\}E_v$ and
analogously $\min\{F\}$ for a finite subset $F\subset L_\QQ$.

Given an element $\ell'\in L'$ we denote by $[\ell']\in H$ its class in $H=L'/L$.
The lattice $L'$ admits a partition parametrized by the group $H$, where for any $h\in H$ one sets
\begin{equation}
\label{eq:Lprime}
L'_h=\{\ell'\in L'\mid [\ell']=h\}\subset L'.
\end{equation}
Note that $L'_0=L$.
Given an $h\in H$ one can define $r_h:=\sum_v l'_v E_v\in L'_h$ the unique element of $L_h'$ such that $0\leq l'_v<1$.
Equivalently, $r_h=\sum_v \{l'_v\} E_v$ for any $\ell'=\sum_v l'_v E_v\in L'_h$, where $0\leq \{\cdot\}<1$
represents the usual fractional part.

We define the rational Lipman cone by
$$\cS_\QQ:=\{\ell'\in L_\QQ \ | \ (\ell',E_v)\leq 0 \ \mbox{for all} \ v\in \V\},$$
which is a cone generated over $\QQ_{\geq 0}$ by $E^*_v$.
Define $\cS':=\cS_\QQ\cap L'$ as the semigroup (monoid) of anti-nef rational cycles of $L'$; it is generated over 
$\mathbb{Z}_{\geq 0}$ by the cycles $E^*_v$. As mentioned in section~\ref{sec:lattice}, any element of $\cS'$ can 
be obtained as the exceptional part of the pull-back of an effective divisor of~$X$.

The Lipman cone $\cS'$ also admits a natural equivariant partition indexed by $H$ by $\cS'_{h}=\cS'\cap L'_h$.
Note the following properties of the Lipman cone:
\begin{enumerate}
\item[(a)] $s_1,s_2\in \cS'_{h}$ implies $s_2-s_1\in L$ and hence $\min\{s_1,s_2\}\in \cS'_h$.
\item[(b)]\label{prop:sh}
for any $s\in L_{{\QQ}}$ the set $\{s'\in \cS'_h \mid s'\not\geq s\}$ is finite,
 since $\{E^*_v\}_v $ have positive entries.
\item[(c)] for any $h$ there exists a unique \textit{minimal cycle} $s_h:=\min \{\cS'_{h}\}$~(see~\ref{ss:GLA} below).
\item[(d)] $\cS'_h$ is a cone with \textit{vertex} $s_h$ in the sense that $\cS'_h=s_h+\cS'_0$.
\end{enumerate}
\subsubsection{{\bf Generalized Laufer's algorithm}}\label{ss:GLA}
\cite[Lemma 7.4]{NemOSZ} For any $\ell'\in L'$ there exists a unique
minimal element $s(\ell')$ of the set $\{s\in \cS' \,:\, s-\ell'\in L_{\geq 0}\}$. It can be obtained by the following algorithm.
Set $x_0:=\ell'$. Then one constructs a computation sequence $\{x_i\}_i$ as follows.
If $x_i$ is already constructed and $x_i\not\in\cS'$ then there exits some $E_{u_i}$ such that $(x_i,E_{u_i})>0$.
Then take $x_{i+1}:=x_i+E_{u_i}$ (for some choice of $E_{u_i}$). Then the procedure after finitely many steps stops,
say at $x_t$, and necessarily $x_t=s(\ell')$.

Note that $s(r_h)=s_h$ and $r_h\leq s_h$, however, in general $r_h\neq s_h$.
(This fact does not contradict the minimality of $s_h$ in $\cS'_h$ since $r_h$ might not sit in $\cS'_h$.)
Also, if $\ell'\in L'_{\leq 0}$, then $s(\ell')=s_{[\ell']}$.

\subsection{Local divisor class group}\label{ss:LDCG} 
Using the exponential exact sequence of $\tilde X$ (and the notation 
$H^1(\tilde X, \cO^*_{{\tilde X}})={\rm Pic}(\tilde X)$ 
and, as usual, $L'=H^2(\tilde X,\ZZ)\simeq H_2(\tilde X,\Sigma,\ZZ)$) we get
\begin{equation}\label{eq:picard} 0\to H^1(\tilde X, \cO_{{\tilde X}})\to {\rm Pic}(\tilde X)\to L'\to 0.\end{equation}
$L$ embeds naturally in both $L'$ and in ${\rm Pic}(\tilde X)$ (in the second one by $\ell \mapsto \cO_{{\tilde X}}(\ell)$).
The group ${\rm Pic}(\tilde X)/L$ is the {\it local divisor class group of }~$(X,0)$, that is, the group of local Weil divisors modulo the local Cartier divisors.
 In particular, we have (the resolution independent) exact sequence
\begin{equation}\label{eq:picard2}
0\to H^1(\tilde X, \cO_{{\tilde X}})\to {\rm Weil}(X)/ {\rm Cartier}(X)\to H_1(\Sigma, \ZZ)\to 0.
\end{equation}
Recall that $h^1(\tilde X, \cO_{{\tilde X}})=p_g(X,0)$ is the
{\it geometric genus}~of the germ $(X,0)$. The singularity $(X,0)$ is called
{\it rational}~if $p_g(X,0)=0$. By the above exact sequences, for rational singularities one has ${\rm Pic}(\tilde X)=L'$ (that is,
any line bundle of $\tilde X$ is determined topologically by its first Chern class) and also, the local divisor class group is
isomorphic with $H=H_1(\Sigma, \ZZ)$. Note that the morphism is induced as follows: take a divisor $D$, then the homology class of its boundary
$\partial D\subset \partial \tilde X=\Sigma$ gives the correspondence. Here we warn the reader that
if $C$ is a reduced Weil divisor germ in $X$, and we set $\pi^*C=\widetilde{C}+\ell'_C$, where $\widetilde{C}$ is the
strict transform and $\ell'_C\in L'$, then usually in this manuscript we set $h=[\ell'_C]$. Hence, since
$\widetilde{C}+\ell'_C=0$ in $H_2(\tilde X, \Sigma,\ZZ)$, the class of $\partial \widetilde{C}$ is $-h$.

\subsection{Multivariable series}
\label{ss:set}
In this subsection we will use the following (slightly more general)
notation. $L$ is a lattice freely generated by base elements $\{E_v\}_{v\in V}$,
$L' $ is an overlattice of the same rank (not necessarily dual of $L$), and we set $H:=L'/L$, a finite abelian group of order $d$.
(The partial ordering is defined as in subsection~\ref{ss:HrepLip}.)
Let $\ZZ[[L']]$ be the $\ZZ$-module
consisting of the $\ZZ$-linear combinations of the monomials $\mathbf{t}^{\ell'}:=\prod_{v\in \V}t_v^{l'_v}$,
where $\ell'=\sum_{v}l'_v E_v\in L'$.
Note that it is a $\ZZ$-submodule of the formal power series in variables $t_v ^{1/d}, t_v^{-1/d}$, $v\in V$.

Consider a multivariable series $S(\tt)=\sum_{\ell'\in L'}a(\ell')\tt^{\ell'}\in \ZZ[[L']]$.
Let $\Supp S(\tt):=\{\ell'\in L' \mid a(\ell')\neq 0\}$ be the support of the
series and we assume the following finiteness condition: for any $x\in L'$
\begin{equation}\label{eq:finiteness}
\{\ell'\in \Supp S(\tt)\mid \ell'\not\geq x\} \ \ \mbox{is finite}.
\end{equation}

Throughout this paper we will use multivariable series in $\ZZ[[L']]$ as well as in $\ZZ[[L'_\I]]$ for any
$\I\subset \V$, where $L'_\I={\rm pr}_I(L')$ is the projection of $L'$ via ${\rm pr}_I:L_{{\QQ}} \to \oplus_{v\in \I} \QQ \langle E_v\rangle$.
For example, if $S(\tt)\in \ZZ[[L']]$ then $S(\tt_{\I}):=S(\tt)|_{t_v=1,v\notin \I}$ is an element of $\ZZ[[L'_\I]]$.
In the sequel we use the notation $\ell'_I=\ell'|_I:= {\rm pr}_I(\ell')$ and $\tt^{\ell'}_\I:=\tt^{\ell'}|_{t_v=1,v\notin \I}$ for any $\ell'\in L'$.
Each coefficient $a_\I(x)$ of $S(\tt_\I)$ is obtained as a summation of certain coefficients $a(y)$ of $S(\tt)$, where $y$ runs over
$ \{\ell'\in \Supp S(\tt)\mid \ell'|_I =x \}$
(this is a finite sum by~\eqref{eq:finiteness}). Moreover, $S(\tt_\I)$ satisfies a similar finiteness property as ~\eqref{eq:finiteness}
in the variables~$\tt_\I$.

Any $S(\tt)\in \ZZ[[L']]$ decomposes in a unique way as $S(\tt)=\sum_h S_h(\tt)$,
where $S_h(\tt):=\sum_{[\ell']=h}a(\ell')\tt^{\ell'}$. $S_h(\tt)$ is called the
$h$-part of $S(\tt)$.
Note that the $H$-decomposition of the \rel series is not well defined. That is, the restriction
$S_h(\tt)|_{t_v=1,v\notin \I}$ of the $h$-part $S_h(\tt)$ cannot be recovered from $S(\tt_\I)$ in general,
since the class of $\ell'$ cannot be recovered from $\ell'|_I$.
Hence, the notation $S_h(\tt_\I)$ (defined as $S_h(\tt)|_{t_v=1,v\notin \I}$) is not ambiguous, but requires certain caution.

\subsection{Counting functions}
\label{ss:countingfunctions}
Given a multivariable series $A(\tt_\I)\in \ZZ[[L'_\I]]$ for $\emptyset\neq \I\subset \V$
(eg., $A(\tt_\I)=S(\tt_\I)$ or $A(\tt_\I)=S_h(\tt_\I)$ for $h\in H$)
two functions associated with the coefficients of $A(\tt_{\I})$ can be considered
(cf. \cite{Nem-PS, LSzPoincare}).
The first one is called the (original) \textit{counting function} (for the motivation of the summation-type see \ref{sec:pg}):
\begin{equation}\label{eq:count1}
Q{(A(\tt_\I))}: L'_\I\longrightarrow \ZZ, \ \ \ \
x_\I\mapsto \sum_{\ell'_\I\ngeq x_\I} a(\ell'_\I).
\end{equation}
It is well defined whenever $A$ satisfies the finiteness condition~\eqref{eq:finiteness}.
If $A(\tt_\I)=S_h(\tt_\I)$ then
$Q{(A(\tt_\I))}$ will also be denoted by~$Q^{(S)}_{h,\I}$.
In this case,
$$
Q{(S(\tt_\I))}=\sum_{h\in H} Q{(S_h(\tt_\I))}=\sum_{h\in H} Q^{(S)}_{h,\I}.
$$

The second function is called the \textit{modified counting function} and it is defined by
\begin{equation}\label{eq:modcount1}
q{(A(\tt_\I))}: L'_\I\longrightarrow \ZZ, \ \ \ \
x_\I\mapsto \sum_{\ell'_\I \prec x_\I} \ a(\ell'_\I),
\end{equation}
where the order relation $\ell'_\I \prec x_\I$ means $\ell'_v<x_v$ for all $v\in \I$.
(A new symbol $x\prec y$ is introduced to avoid ambiguity with $x<y$, which means $x\leq y$ and $x\neq y$.)

The inclusion-exclusion principle connects the two counting functions, namely,
\begin{equation}\label{eq:InEx}
Q^{(S)}_{h,\I}(x)=\sum_{\emptyset \not=\J\subset \I} \ (-1)^{|\J|+1} q^{(S)}_{h,\J}(x).
\end{equation}
Both counting functions (and hence identity~\eqref{eq:InEx} as well) can be extended to $x\in L'$
via the projection $L'\to L'_\I$, that is, $Q^{(S)}_{h,\I}(x)=Q^{(S)}_{h,\I}(x_\I)$.

\subsection{\bf Periodic constants \texorpdfstring{\cite{LNehrhart}}{20}}\label{ss:pc}
Let $S(\tt) \in \ZZ[[L']]$ be a series satisfying the finiteness condition~\eqref{eq:finiteness}
and consider its $h$-part $S_h(\tt)$ for a fixed $h \in H$.
Let $\mathcal{K}\subset L'\otimes\mathbb{R}$ be a real closed cone whose affine closure is top dimensional.
Assume that there exist $\ell'_* \in \mathcal{K}$ and a finite index sublattice $\widetilde{L}$ of $ L$ and a
quasi-polynomial $\qp{Q}^{{\mathcal K},(S)}_{h,\V}(x)$
defined on $\widetilde{L}$ such that
\begin{equation}
\label{eq:qpol}
\qp{Q}_{h,\V}^{{\mathcal K},(S)}(\ell) = Q_{h,\V}^{(S)}(r_h+\ell)
\end{equation}
whenever $\ell\in (\ell'_* +\mathcal{K})\cap \widetilde{L}$ ($r_h\in L'$ is defined similarly as in~\ref{ss:HrepLip}).
Then we say that the counting function $Q_{h,\V}^{(S)}$ (or just $S_h(\mathbf{t})$)
\textit{admits a quasi-polynomial} in $\mathcal{K}$, namely $\widetilde{L}\ni \ell\mapsto
\qp{Q}_{h,V}^{\mathcal{K},(S)}(\ell)$.
In this case, we define the \textit{periodic constant} of $S_h(\tt)$ associated with $\mathcal{K}$ by
\begin{equation}\label{eq:PCDEF}
\mathrm{pc}^{\mathcal{K}}(S_h(\tt)) := \qp{Q}_{h,\V}^{{\mathcal K},(S)}(0) \in \ZZ.
\end{equation}
The number $\mathrm{pc}^{\mathcal{K}}(S_h(\tt))$ is independent of the
choice of $\ell'_*$ and of the finite index sublattice
$\widetilde L\subset L$.

The same construction can be applied to the modified counting function as well. Namely, if
$q_{h,\V}^{(S)}$ admits a quasi-polynomial in $\mathcal{K}$, say $\qp{q}_{h,\V}^{\, {\mathcal K},(S)}(x)$, then the
{\em modified periodic constant} of $S_h(\tt)$ (or simply the periodic constant of $q_{h,V}^{{\mathcal K},(S)}$)
associated with $\mathcal{K}$ is defined by
\begin{equation}\label{eq:MPCDEF}
\mathrm{mpc}^{\mathcal{K}}(S_h(\tt)) :=\qp{q}_{h,\V}^{\, {\mathcal K},(S)}(0).
\end{equation}
In some cases we might drop the indices $\mathcal{K}$ or $S$ if there is no ambiguity regarding them.

Given any $\I\subset \V$ the natural group homomorphism ${\rm pr}_I:L'\to L'_I$ preserves the lattices $L\to L_I$
and hence it induces a homomorphism $H\to H_I:=L'_I/L_I$, denoted by $h\mapsto h_I$. (However, note that even if
$L'$ is the dual of $L$ associated with a form $(\,,\,)$,
$L_I'$ usually is not a dual lattice of $L_I$, it is just an overlattice. This fact motivates the general setup of the present subsection.)
In this \textit{projected context} one can define again
the (modified) periodic constant associated with the \rel series $S_h(\tt_I)$ from the previous paragraph exchanging
$\V$ (resp. $\tt$, $r_h$) by $\I$ (resp. $\tt_I$, $r_{h_I}$).

\begin{exam}
\label{ex:pc}
\mbox{}
\begin{enumerate}
 \item\label{ex:pc1}
 If $S_h(\tt)$ is a (Laurent) polynomial in $\ZZ[[L']]$, that is, $\Supp(S_h(\tt))$ is finite, then
 its counting function $Q^{(S)}_{h,V}$ is constant for large enough values and this constant
 equals the sum of its coefficients,
 that is, $S_h(1)$. Hence  $\qp{Q}^{\, {\mathcal K},(S)}_{h,V}$ eg. for ${\mathcal K}=(\mathbb{R}_{\geq 0})^{|V|}$
 is the constant map~$S_h(1)$ and thus its periodic constant exists and equals~$S_h(1)$.
(The periodic constant for ${\mathcal K}=(\mathbb{R}_{\leq 0})^{|V|}$ also exists and it equals zero.)
 \item\label{ex:pc2}
 The periodic constant of one-variable series was introduced in~\cite{Ok,NOk} as follows.
 For simplicity, we assume that $L=L'\simeq\ZZ$ and let $S(t)=\sum_{\ell\geq 0}c_\ell t^\ell \in \mathbb{Z}[[t]]$ be a
 formal power series in one variable. If for some $p\in \mathbb{Z}_{>0}$ the counting function
 $Q^{(p)}(n):=\sum_{\ell=0}^{pn-1}c_\ell$ is a polynomial $\qp{Q}^{(p)}$ in $n$, then the constant term
 $\qp{Q}^{(p)}(0)$ is independent of $p$ and it is called the periodic constant $\mathrm{pc}(S)$ of the series~$S$.
(Here $\widetilde{L}=p\ZZ$ and the cone is automatically the `positive cone'.)
 \item\label{ex:pc3}
 Assume that the coefficients $\ell \mapsto c_{\ell}$ of the one-variable series $S(t)$ is given by a Hilbert function,
 which admits a Hilbert polynomial $H(\ell)$ with $c_{\ell}=H(\ell)$ for $\ell\gg 0$. Then
 a computation shows that the periodic constant
of the
 regularized series $S^{reg}(t)=\sum_{\ell\geq 0}H(\ell)t^{\ell}$ is zero, hence
 $\pc(S)=\pc(S-S^{reg})+\pc(S^{reg})=(S-S^{reg})(1)$, measuring the difference between the Hilbert function and
 Hilbert polynomial. Eg., if $S(t)$ is the generating series of a numerical semigroup
 ${\mathfrak{S}}$
 with finite complement in $\ZZ_{\geq 0}$, then ${\rm pc}(S)=-
 |\ZZ_{\geq 0}\setminus \mathfrak{S}|$.
\end{enumerate}
\end{exam}

\subsection{Poincar\'e series of surface singularities}\label{s:Pseries}
\subsubsection{}\label{ss:uac}
We fix a good resolution $\pi$ of $X$. Consider $c:Y\to X$, the universal abelian covering of $(X,0)$, let
$\tilde Y$ be the normalized pull-back of $\pi$ and $c$, and denote by $\pi_Y$ and $\tilde{c}$
the induced maps by the pull-back completing the following commutative diagram.
\begin{equation}
\label{eq:diagram}
\xymatrix{
\ar @{} [dr] | {\#}
\tilde{Y} \ar[r]^{\tilde{c}} \ar[d]_{\pi_Y} & \tilde{X} \ar[d]^{\pi} \\
Y \ar[r]_{c} & X
}
\end{equation}

We define the following $H$-equivariant $L'$-indexed divisorial filtration of the local ring $\mathcal{O}_{Y}$:
for any given $\ell'\in L'$ set
\begin{equation}
\label{eq:Ffiltration}
\mathcal{F}(\ell'):=\{g\in \cO_{Y} \mid \div (g\circ \pi_Y)\geq \tilde c^* (\ell')\}.
\end{equation}
It is worth to mention
that the pull-back $\tilde c^*(\ell')$ is an integral cycle in $\tilde Y$ for any $\ell'\in L'$,
cf.~\cite[Lemma 3.3]{Nem-CLB}. The natural action of $H$ on $Y$ induces an action on $\cO_{Y}$ as follows:
$h\cdot g(y)=g(h\cdot y)$, $g\in \cO_{Y}$, $h\in H$. This action decomposes
$\cO_{Y}$ as $\oplus_{\lambda\in \hat{H}} (\cO_{Y})_{\lambda}$ according to the characters
$\lambda \in \hat{H}:={\rm Hom}(H,\CC^*)$, where
\begin{equation}
\label{eq:Heigenspaces}
(\cO_{Y})_{\lambda}:=\{g\in \cO_{Y} \mid g(h\cdot y)=\lambda (h)g(y),\ \forall y\in Y, h\in H\}.
\end{equation}
Note that there exists a natural isomorphism $\theta:H\to \hat{H}$ given by
$h\mapsto \exp(2\pi \sqrt{-1} (\ell',\cdot ))\in \Hom(H,\CC^*)$, where $\ell' $ is any element of $L'$ with
$h=[\ell']$. In order to simplify our notations we will use the notation $(\cO_Y)_h$ for $(\cO_Y)_{\theta(h)}$ (and similarly for any
linear $H$-representation).

The subspace $\mathcal{F}(\ell')$ is invariant under this
action and $ \mathcal{F}(\ell')_h=\mathcal{F}(\ell')\cap (\cO_{Y})_h$. Thus, one can define the \textit{Hilbert function}
 $\mathfrak{h}(\ell')$ for any $\ell'\in L'$
as the dimension of the $\theta([\ell'])$-eigenspace $(\cO_{Y}/\mathcal{F}(\ell'))_{[\ell']}$.
The corresponding multivariable \textit{Hilbert series} is
\begin{equation}
\label{eq:Hilbert}
H(\mathbf{t})=\sum_{\ell'\in L'} \mathfrak{h}(\ell')\mathbf{t}^{\ell'}\in \mathbb{Z}[[L']].
\end{equation}

The $H$-eigenspace decomposition of $\tilde c_{*}(\cO_{\tilde Y})$ is given by (see~\cite{Nem-PS,Okuma-abelian})
\begin{equation}\label{eq:decUAC}
\tilde c_{*}(\cO_{\tilde Y})=\bigoplus_{h\in H}\cO_{\tilde X}(-r_h) \  \ \mbox{with} \
\cO_{\tilde X}(-r_h)=(\tilde c_{*}(\cO_{\tilde Y}))_h,
\end{equation}
where $\cO_{\tilde X}(\ell')$ is the only line bundle $\mathcal L$ on $\tilde X$ satisfying
$\tilde c^*\mathcal L=\cO_{\tilde Y}(\tilde c^*(\ell'))$ (see~\cite[3.5]{Nem-CLB}) and $r_h$ is the
representative of $h$ as in section~\ref{ss:HrepLip}.
(As a word of caution: $r_h$ is a $\QQ$-divisor in $\tilde X$, and the notation
$\cO_{\tilde{X}}(-r_h)$ here is different from the one used by Sakai in~\cite{Sakai84}.)

\subsection{{\bf Some useful expressions for $\mathfrak{h}(\ell')$}}\label{ss:vanishingth}
For any $\ell'=\ell+r_h>0$
one has the following alternative expression of the Hilbert function (cf.~\cite[Corollary.~4.2.4]{Nem-PS})
\begin{equation}\label{eq:hdim}
\mathfrak{h}(\ell')=\dim \frac{H^0(\tilde X, \cO_{\tilde X}(-r_h))}{H^0(\tilde X, \cO_{\tilde X}(-r_h-\ell))}.
\end{equation}
Furthermore, for
 $\ell\in L$ effective and $\ell'_1\in L'$ one has the
 restriction exact sequence
\begin{equation}
\label{eq:restriction2}
0\to \cO_{\tilde X}(-\ell-\ell'_1)\to \cO_{\tilde X}(-\ell'_1)\to \cO_{\ell}(-\ell'_1)\to 0.
\end{equation}
Applying~\eqref{eq:restriction2} to $\ell'_1=r_h$ and writing $\ell'=\ell+r_h$ one obtains
$$\mathfrak{h}(\ell') -\chi(\cO_{\ell}(-r_h))+h^1(\cO_{{\tilde X}}(-\ell'))-h^1(\cO_{{\tilde X}}(-r_h))=0$$
or equivalently,
\begin{equation}\label{eq:hell}
\mathfrak{h}(\ell') =\chi(\ell')-h^1(\cO_{{\tilde X}}(-\ell'))-\chi(r_h)+ h^1(\cO_{{\tilde X}}(-r_h)).
\end{equation}
On the other hand, in the context of the generalized Laufer's algorithm, since $x_{i+1}=x_i+E_{u_i}$
(see~\ref{ss:GLA}) using equation~\eqref{eq:restriction2} repeatedly and also $(x_i,E_{u_i})>0$ one obtains
\begin{equation}\label{eq:sell'}
h^1(\cO_{{\tilde X}}(-\ell'))-\chi(\ell')=h^1(\cO_{{\tilde X}}(-s(\ell')))-\chi(s(\ell')).
\end{equation}
Note that, if $\ell'=r_h$ then $s(r_h)=s_h$, hence
\begin{equation}\label{eq:rhsh}
h^1(\cO_{{\tilde X}}(-r_h))-\chi(r_h)=h^1(\cO_{{\tilde X}}(-s_h))-\chi(s_h).
\end{equation}

\subsubsection{\bf{The series $P(\mathbf{t})$ and $Z(\mathbf{t})$}}\label{ss:PZ}
Campillo, Delgado and Gusein-Zade~\cite{CDG-Poincare,CDGZ-PSuniversalAC} have defined a different series too,
it is called the
\textit{multivariable analytic Poincar\'e series} $P(\mathbf{t})=\sum_{\ell'\in L'}\mathfrak{p}(\ell')\mathbf{t}^{\ell'}$.
It can be defined by the identity
\begin{equation}
\label{eq:analPoincare}
P(\mathbf{t})=-H(\mathbf{t})\cdot \prod_{v\in \V}(1-t_v^{-1}).
\end{equation}
Although by considering $P$ instead of $H$ one might in principle loose some information
(because $\prod_{v}(1-t_v^{-1})$ is a zero divisor), in fact, $P$ and $H$ do contain the same amount of
information (see~\cite[p.~50]{CDG-Poincare}). See~\cite{Nem-CLB} for the next concrete inversion identity
\begin{equation}\label{eq:hinv}
\mathfrak{h}(\ell')=\sum_{\ell\in L,\ell\ngeq 0} \mathfrak{p}(\ell'+\ell).
\end{equation}

The \textit{multivariable topological Poincar\'e series} (cf.~\cite{CDG-Poincare,CDGZ-PSuniversalAC,Nem-CLB})
is the Taylor expansion $Z(\tt)=\sum_{\ell'} z(\ell')\tt^{\ell'} \in\ZZ[[L']]$
at the origin of the rational \textit{zeta function}
\begin{equation}\label{eq:1.1}
f(\tt)=
\prod_{v\in \V} (1-\tt^{E^*_v})^{\text{val}(v)-2}.
\end{equation}
In both cases one can consider their $H$-decomposition $P(\mathbf{t})=\sum_{h\in H}P_h(\mathbf{t})$ and
$Z(\mathbf{t})=\sum_{h\in H}Z_h(\mathbf{t})$ as well.
The supports of $P(\tt)$ and $Z(\tt)$ are in the Lipman cone
$\cS'$ (for $P(\tt)$ see eg.~\cite[pp.~7-8]{Nem-PS} while for $Z(\tt)$ follows from the definition).
Since by~\ref{ss:HrepLip}(b)
for any $x\in L'$ the set
$\{\ell'\in \cS'\,:\, \ell'\not\geq x\}$ is finite, both $P(\tt)$ and $Z(\tt)$ satisfy the finiteness
condition~\eqref{eq:finiteness}, hence the corresponding counting functions are also well-defined, cf. ~\ref{ss:set}
(and the sum in (\ref{eq:hinv}) is also finite).

\subsubsection{\bf{Reduced versions}}
\label{ss:red}

Sometimes it is natural to restrict the variables of the series to a particular subset of the vertices of the resolution graph,
mostly supporting an intrinsic geometric meaning.

Following section~\ref{ss:set}, we fix a subset $\I\subset \V$ and the projection
 $L'\to L'_\I\subset \QQ\langle E_v\rangle_{v\in I}$.
Then for any $h\in H$ we define the
\textit{reduced topological Poincar\'e series} as $Z_h(\tt_{\I}):=Z_h(\tt)|_{t_v=1,v\notin \I}\in \ZZ[[L'_\I]]$.
(Recall that the reduction procedure must be considered
after the $H$-decomposition, since in general the $H$-decomposition of $Z(\tt_\I)$ is not well defined.)

Both the multivariable Hilbert series defined in~\eqref{eq:Hilbert} as well as the analytic Poincar\'e
series~\eqref{eq:Hilbert} can be reduced to $\I\subset\V$. Alternatively, one could also define a reduced
$L'_I$-divisorial filtration following~\eqref{eq:Ffiltration} as
\begin{equation}\label{eq:redfilt}
\mathcal{F}(\ell'_\I):=\{g\in \cO_{Y} \mid \div (g\circ \pi_Y)\geq \tilde c^* (\ell'_\I)\}.
\end{equation}
Then, for any $h\in H$ one could define an alternative Hilbert series 
whose $\tt^{\ell'_\I}$- coefficient $\mathfrak{h}_{h,I}(\ell'_\I)$
 is defined as the dimension of the $h$-eigenspace $(\cO_{Y}/\mathcal{F}(\ell'_\I))_{h}$.
Using relation~\eqref{eq:analPoincare} one could alternatively give a definition of a reduced analytic Poincar\'e series as well.
It was shown in
\cite[Theorem 6.1.7]{Nem-PS} that both approaches result in the same series $H_h(\tt_\I)$ and $P_h(\tt_\I)$.

\subsubsection{\bf{Surgery formulae}} \label{sec:surgform1}
Let us consider $\Gamma$ the dual graph of a good resolution of $(X,0)$ and fix a subset $\I\subset \V$ of its vertices.
The set of vertices $\V\setminus \I$ determines the connected full subgraphs $\{\Gamma_k\}_k$ with vertices
$\V(\Gamma_k)$, i.e. $\cup_k\V(\Gamma_k)=\V\setminus \I$. We associate with any $\Gamma_k$ the lattices
$L(\Gamma_k)$ and $L'(\Gamma_k)$ as well, endowed with the corresponding intersection forms.

Then for each $k$ one considers the inclusion operator $j_k:L(\Gamma_k)\to L(\Gamma)$,
$E_v(\Gamma_k)\mapsto E_v(\Gamma)$, identifying naturally the corresponding $E$-base elements associated with the two
graphs. This preserves the intersection forms.

Let $j_{k}^*:L'(\Gamma)\to L'(\Gamma_k)$ be the dual operator, defined by
$j_{k}^*(E^*_{v}(\Gamma))=E^*_{v}(\Gamma_k)$ if $v\in\V(\Gamma_k)$, and $j_{k}^*(E^*_{v}(\Gamma))=0$ otherwise.
Note that $j^*_{k}(E_v(\Gamma))=E_v(\Gamma_k)$ for any $v\in \V(\Gamma_k)$.
Then we have the projection formula
$(j^*_{k}(\ell'), \ell)_{\Gamma_k}=(\ell',j_{k}(\ell))_{\Gamma}$ for any $\ell'\in L'(\Gamma)$ and $\ell\in L(\Gamma_k)$,
which also implies that
\begin{equation}\label{eq:proj\ZK}
j^*_{k}(\ZK)=Z_{K}(\Gamma_k),
\end{equation}
where $\ZK$ (resp. $\ZK(\Gamma_k)$) is the canonical cycle associated with $\Gamma$ (resp. $\Gamma_k$).

Similarly, the next formula holds  for the corresponding minimal cycles $s_h$:

\begin{lemma}[\cite{LSzPoincare}]\label{lem:projs_h}
For any $h\in H$ one has $j^*_{k}(s_{h})=s_{[j^*_{k}(s_{h})]}\in L'(\Gamma_k)$.
\end{lemma}

Moreover, if $Q^\Gamma$ and $Q^{\Gamma_k}$ denote the counting functions associated with the topological Poincar\'e
series of the corresponding graphs, then one has the following surgery formula.

\begin{thm}[{\cite[Theorem 3.2.2]{LNN}}]
\label{thm:surgform}
For any $\ell'=\sum_va_vE^*_v$ with $a_v\gg 0$ and with the notation $[\ell']=h$
 one has the identity
\begin{equation}\label{eq:surgform}
Q^\Gamma_{h}\,(\ell')=Q^\Gamma_{h,\I}\,(\ell')+\sum_k \
Q^{\Gamma_k}_{[j^*_k(\ell')]}(j^*_k(\ell')).
\end{equation}
\end{thm}
More details regarding the above setting can be found in~\cite{LNN}.
\subsubsection{\bf Geometric genera, Seiberg-Witten invariants and the CDGZ-identity}
\label{sec:pg}
By the next package of properties of these series we show that
they encode important analytical, respectively topological, invariants of the analytic type,
resp. of the link.

As a consequence of the inversion identity~\eqref{eq:hinv}, the counting function for $P_h(\tt)$ is the
Hilbert function $\mathfrak{h}_h(x)$. In particular, in the sequel we will use the notation
 $\mathfrak{h}_h(x)$ for $Q^{(P)}_h(x)$.
 Therefore, in the sequel we will use the simplified notation $Q_h(x)=Q^{(Z)}_h(x)$
for the topological counting function of~$Z_h(\tt)$, if there is no danger of ambiguity.

Using~\eqref{eq:hell} and the Grauert-Riemenschneider vanishing
(see also~\cite[Prop. 3.2.4]{Nem-CLB}) we obtain that for any $\ell'=r_h+\ell\in \ZK+\cS'$
(where $\ell\in L$ and $h=[\ell']$) the Hilbert function can be written in the form
\begin{equation}\label{eq:SUMh}
\mathfrak{h}(\ell')=\chi(\ell')-\chi(r_h)+h^1(\tilde X, \cO_{\tilde X}(-r_h)).
\end{equation}
The right hand side is a multivariable polynomial in $\ell\in L$ by the Riemann-Roch formula, cf.~\ref{sec:lattice}.
This means that $P_{h}(\tt)$ admits a (quasi-)polynomial in $\cS'_{\RR}:=\cS'\otimes\RR$,
hence by~\eqref{eq:PCDEF} its periodic constant is
$\mathrm{pc}^{\cS'_{\RR}}(P_h(\tt))=h^1(\tilde X, \cO_{\tilde X}(-r_h))$. Recall that
 $\{h^1(\tilde X, \cO_{\tilde X}(-r_h))\}_h$ are the equivariant geometric genera of~$(X,0)$,
the $H$--decomposition of the geometric genus of the universal abelian covering.

Similarly, in the case of the topological series we have an analogous identity too:  the only difference is that
the equivariant genera are replaced by the Seiberg-Witten invariants of the link. More precisely,~\cite{NJEMS}
shows that for any $\ell'\in \ZK+ int(\cS')$ (where $int(\cS')=\ZZ_{>0}\langle E^*_v\rangle_{v\in \V}$) the counting 
function of $Z_h(\tt)$ has the form
\begin{equation}\label{eq:SUMQ}
{Q}_{h}(\ell')=
\chi(\ell')-\chi(r_h)+\mathfrak{sw}_h^{\text{norm}}(\Sigma),
\end{equation}
where $\mathfrak{sw}_h^{\text{norm}}(\Sigma)$ denotes the {\it $r_h$-normalized Seiberg-Witten invariant} of the link
$\Sigma$ for $h\in H$. Thus $Z_h(\tt)$ admits the (quasi-)polynomial
\begin{equation}\label{eq:SUMQP}
\qp{Q}_{h}(\ell)=\chi(\ell+r_h)-\chi(r_h)
+\mathfrak{sw}_h^{\text{norm}}(\Sigma),
\end{equation}
in the cone $\cS'_{\mathbb{R}}$ and $\mathrm{pc}^{\cS'_{\RR}}(Z_h(\tt))=\mathfrak{sw}_h^{\text{norm}}(\Sigma)$.
For more details about the Seiberg-Witten invariants of links of singularities we refer the reader
to~\cite{NemOSZ,NJEMS,LNehrhart,LNN}.

The above surprising similarity was conceptually formulated under the name \textit{Seiberg-Witten Invariant Conjecture}
by N\'emethi and Nicolaescu~\cite{NN-SWI}, which says that for \textit{sufficiently nice} normal surface
singularities with rational homology sphere link, the equivariant geometric genera are topological, and they are
given by the Seiberg-Witten invariant of the link. For several cases, eg. for rational singularities, the conjecture is true, see~\cite{NemOSZ,Nem-CLB,Nem-ICM}.

The conjecture can be raised to the level of the series as well. This is the subject of the \textit{CDGZ-identity},
named after Campillo, Delgado and Gusein-Zade thanks to their contributions in~\cite{CDG-Poincare} and~\cite{CDGZ-PSuniversalAC}.
Namely, they proved the following:
\begin{thm}[CDGZ-identity \cite{CDG-Poincare}]\label{thm:CDGZ}
A rational surface singularity $(X,0)$ and any of its resolution $\pi$ satisfy the identity
\begin{equation}\label{CDGZ}
P(\tt)=Z(\tt).
\end{equation}
\end{thm}
The fourth author in~\cite{Nem-PS},~\cite{Nem-CLB} extended the result for larger
families of normal surface singularities with rational homology sphere link.
The largest family for which the CDGZ-identity holds is the family of
splice-quotient singularities with their special resolutions satisfying the end-curve conditions, cf.~\cite{Nem-ICM}.

\begin{cor}\label{cor:sharp} If $(X,0)$ is rational then the following facts hold.
\begin{enumerate}[leftmargin=1cm,itemindent=0cm,labelwidth=\itemindent,labelsep=.2cm,align=left,label={$(\alph*)$}]
 \item\label{cor:sharp1} 
The quasi-polynomials of the analytic and of the topological Poincar\'e series coincide:
for $\ell'$ with sufficiently large $E^*$-coefficients one has
$$\mathfrak{h}_h(\ell')=
\chi(\ell')-\chi(r_h)+h^1(\tilde X, \cO_{\tilde X}(-r_h))=
\chi(\ell')-\chi(r_h)+\mathfrak{sw}_h^{\text{norm}}(\Sigma)=\widetilde{Q}_h(\ell').$$
This combined with~\eqref{eq:rhsh} and with the Lipman's vanishing $h^1(\cO_{{\tilde X}}(-s_h))=0$ \cite{Lipman} gives
$$h^1(\tilde X, \cO_{\tilde X}(-r_h))=
\mathfrak{sw}_h^{\text{norm}}(\Sigma)= \chi(r_h)-\chi(s_h).$$
 \item\label{cor:sharp2}
For any $\ell'=\ell+r_h\in Z_K+\mathcal{S}'$ one has $\widetilde{Q}_h(\ell)=Q_h(\ell+r_h)$.
\end{enumerate}
\end{cor}
\begin{proof} 
Part~\ref{cor:sharp1} follows from the validity of the Seiberg-Witten Invariant Conjecture for rational singularities
 (or from Theorem~\ref{thm:CDGZ}) and the identities~\eqref{eq:SUMh} and~\eqref{eq:SUMQ}. Part~\ref{cor:sharp2} follows from~\eqref{eq:SUMh}
valid for any $\ell'\in Z_K+\mathcal{S}'$.
\end{proof}

\subsection{The relative theory}

\subsubsection{\bf Poincar\'e series for curves}\label{sec:relabsP}
Let $(C,0)$ be an (abstract) reduced curve germ and let $(C,0)=\bigcup_{i\in \I} (C_i,0)$ be its irreducible decomposition with
finite index set $\I$. Let $\gamma_i:(\mathbb{C},0)\to (C_i,0)$, $t\mapsto \gamma_i(t)$, be a normalization of the branch
$(C_i,0)$ and for any germ $g\in \mathcal{O}_{C,0}$ we consider its value $v(g):=(v_i(g))_{i\in \I}$ defined by
$v_i(g):=\ord_{t}(g(\gamma_i(t)))$ for any $i\in \I$. For any $\ell\in \ZZ^{|\I|}$ one associates the ideal
$\mathcal{F}_C(\ell):=\{g\in \mathcal{O}_{C,0} \mid v(g)\geq \ell\}$ and the multivariable
 Hilbert series $H_C(\tt_\I):=\sum_{\ell}\mathfrak{h}_C(\ell) \tt_\I^\ell$, where the Hilbert function is defined by
$\mathfrak{h}_C(\ell):=\dim_{\CC} \mathcal{O}_{C,0}/\mathcal{F}_C(\ell)$.
The Poincar\'e series associated with $(C,0)$ is defined as $P_C(\tt_\I):=-H_C(\tt_\I)\cdot\prod_{i\in \I}(1-t_i^{-1})$.
Thus, if we write $P_C(\tt_\I)=\sum_\ell \mathfrak{p}_C(\ell)\tt_\I^\ell$ then its coefficients can be expressed
by the Hilbert functions as
\begin{equation}\label{eq:coefC}
\mathfrak{p}_C(\ell):=\sum_{J\subset \I}(-1)^{|J|+1} \mathfrak{h}_C(\ell+1_J).
\end{equation}

Moreover, Lemma \ref{lem:hilbinc} and~\eqref{eq:coefC} imply the following property.
\begin{lemma}[cf.~\cite{CDK}]\label{lem:Ppoly}
For $|\I|\geq 2$ the Poincar\'e series $P_C(\tt_\I)$ is a polynomial.
\end{lemma}

Let $S_C$ be the value semigroup with its conductor $c\in \NN^{|\I|}$, and we use the notation $1_J=\sum_j 1_j$ where
$1_j$ is the $j$-th vector of the canonical basis of $\ZZ^{|\I|}$ and, for convenience $1_J=0$ if $J=\emptyset$.
 Then the Hilbert function satisfies the following useful properties:
\begin{lemma}[{\cite[Lemma 3.4]{CDK}}, see also \cite{Moy}]\label{lem:hilbinc} For any $i\in I$ fixed,
$\mathfrak{h}_C(\ell+1_i)=h(\ell)+1$ if and only if there exists $s\in S_C$ with $s_i=\ell_i$ and $s_j\geq \ell_j$ for each $j$.
\end{lemma}
The delta invariant $\delta(C)$ of $(C,0)$ is defined as follows.
Let $\gamma:(C,0)^{\widetilde{}}\to(C,0)$
be the normalization, where $(C,0)^{\widetilde{}}$ is the corresponding multigerm.
Then $\delta(C)$ is defined as $\dim_{{\CC}}\gamma_*\cO_{(C,0)^{\widetilde{}}}/\cO_{(C,0)}$.

\begin{lemma}[\cite{MZG}]
For any $\ell\in \NN^{|\I|}$ with $\ell\geq c$ one has
\begin{equation}\label{eq:sgpid}
\mathfrak{h}_C(\ell)=|\ell|-\delta(C),
\end{equation}
where the norm of $\ell=(\ell_i)_{i\in \I}\in \NN^{|\I|}$ is $|\ell|:=\sum_{i\in \I}\ell_i$.
\end{lemma}

\subsubsection{\bf The relative CDGZ-identity}\label{ss:rCDGZ}
Let now $(C,0)$ be a reduced curve germ embedded in a normal surface singularity~$(X,0)$.
We fix a good embedded resolution $\pi:\tilde X \to X$ of $C\subset X$, and consider $L$ and $L'$
the corresponding lattices associated with $\pi$. Denote by $\widetilde C$ the strict transform of $C$
and $\I_C\subset \V$ be the set of irreducible exceptional divisors which intersect $\widetilde C$.

Moreover, we further assume that the resolution $\pi$ satisfies the following condition
\begin{equation}\label{rescond}
\mbox{every component } E_v \ (v\in \I_C\subset \V) \mbox{ of } E \mbox{ intersects only one component }
\widetilde{C}_v \mbox{ of } \widetilde C.
\end{equation}
This technical condition is not essential, but will simplify the exposition.

In the topological setting we define the \textit{relative multivariable topological Poincar\'e series}
associated with $C\subset X$ as the modified series
\begin{equation}\label{eq:rtPs}
Z^{C}(\tt):=Z(\tt)\cdot \prod_{v\in \I_C}(1-\tt^{E^*_v}),
\end{equation}
where $\I_C\subset \V$ is defined above. 
According to section~\ref{ss:set}, one can consider its $H$-decomposition, namely,
$Z^{C}(\tt)=\sum_{h\in H} Z^{C}_h(\tt)$ and the reduced series
$Z^C_h(\tt_\I):=Z^{C}_{h}(\tt)|_{t_v=1, v\notin \I}$.
Then Campillo, Delgado and Gusein-Zade proved the following identity relating
$P_C(\tt_I)$, the abstract Poincar\'e series of $C$, and $Z^C_{0}(\tt_I)$, the 0-part of the
relative multivariable topological Poincar\'e series,
see also~\cite[Corollary 9.4.1]{Nem-PS} in the equivariant context.
\begin{thm}[Relative CDGZ-identity \cite{CDGZ-PScurves}]\label{relCDGZ}
For any reduced curve germ $(C,0)$ embedded into the rational singularity $X$ one has
$$P_C(\tt_I)=Z^C_{0}(\tt_I).$$
\end{thm}

\begin{rem}
As mentioned above, the extra assumption~\eqref{rescond} on the resolution $\pi$ is unnecessary for
the above identity, cf.~\cite[Theorem 2]{CDGZ-PScurves}, however it makes the previous and
upcoming statements more transparent.
\end{rem}

\section{The \texorpdfstring{$\kappa$}{kappa}-invariant for a~\texorpdfstring{\RHS}{RHS} singularity}
\label{sec:kappa-inv}
In this section we follow the notation from section~\ref{s:Pseries}. In a previous article~\cite{kappa}
the authors defined the following invariant associated with a resolution $\pi:\tilde X\to X$.

\begin{dfn}
For any $\ell'\in L'$ set
\begin{equation}\label{kappa_big}
\kappa_X(\ell'):=\mathfrak{h}(Z_K+\ell')=\dim_\CC \Big(\frac{\cO_{Y}}{\mathcal{F}(Z_K+\ell')}\Big)_{[Z_K+\ell']}.
\end{equation}
\end{dfn}
By the inversion identity~\eqref{eq:hinv} it follows that
\begin{equation}\label{eq:kapcf}
\kappa_X(\ell')=\sum_{\ell\in L,\ell\ngeq 0} \mathfrak{p}(Z_K+\ell'+\ell)
=\sum_{\smallmatrix \ell''\ngeq Z_K+\ell' \\ [\ell'']=[Z_K+\ell'] \endsmallmatrix} \mathfrak{p}(\ell'')
\end{equation}
is given by the counting function of coefficients associated with the analytic Poincar\'e series $P_{[Z_K+\ell']}(\tt)$.

Next, assume that $(C,0)$ is a reduced curve germ on $(X,0)$. The total transform of $C$ on $\tilde X$ is written as
$\pi^* C =\widetilde C+\ell'_C$, where $\ell'_C\in \mathcal{S}'$ is the cycle supported on the exceptional divisor,
 $\widetilde C$ is the strict transform of $C$, and $\pi^*C$ as element of $H_2(\tilde X, \partial \tilde X, {\mathbb Z})$ is trivial.
It turns out that the invariant $\kappa(\ell'_C)$ is independent
of the resolution $\pi$ as long as one chooses a good resolution of~$(X,C)$~(see~\cite[Corollary 3.4]{kappa}).

This justifies the following definition of the $\kappa$-invariant of a reduced curve germ $C\subset X$ in a
\RHS surface singularity.

\begin{dfn}
\label{def:kappa}
Let $(X,0)$ be a \RHS surface singularity and $C\subset X$ a reduced curve germ. The $\kappa$-invariant of the pair $C\subset X$
is defined as
\begin{equation}
\kappa_X(C)=\kappa_X(\ell'_C)=\dim_\CC \Big(\frac{\cO_{Y}}{\mathcal{F}(Z_K+\ell'_C)}\Big)_{[Z_K+\ell'_C]}.
\end{equation}
\end{dfn}
One of the main results of~\cite{kappa} relates $\kappa_X(C)$ and the $\delta$-invariant of $(C,0)$ as follows.

\begin{thm}[{\cite[Theorem 1.1 \& 1.2]{kappa}}]
\label{thm:mainkappa}
If $(X,0)$ is a rational surface singularity and $C\subset X$ a reduced curve germ, then
\begin{equation}
\label{eq:mainkappa}
\kappa_{X}(C)=\delta(C)=\chi(-\ell_C')-\chi(s_{-[\ell'_C]}).
\end{equation}
\end{thm}
Moreover, \cite[Theorem 4.2(3)]{kappa} in the context of an arbitrary normal surface singularity
describes those terms which obstruct the identity $\kappa_{X}(C)=\delta(C)$.
They all vanish when $(X,0)$ is rational. For an arbitrary $(X,0)$, \cite[Lemma 3.3]{kappa} provides the identity
$$\kappa_X(C)=\chi(-\ell_C')-\chi(r_{[Z_K+\ell_C']})+h^1(\cO_{\tilde X}(-r_{[Z_K+\ell_C']})).$$
Here $\chi(-\ell_C')$ and $\chi(r_{[Z_K+\ell_C']})$ depend on the embedded topological type of the pairs
$C\subset X$, while $h^1(\cO_{\tilde X}(-r_{[Z_K+\ell_C']}))$ depends both on the analytic type of $(X,0)$
and on the embedded topological type of the pair. Hence, in this generality, it is hard to expect that 
$\kappa_X(C)$ might be computable either from the embedded topological type of the pair,
or from the analytic type of the abstract curve $(C,0)$.

Furthermore, in general, $\delta(C)$ (as an abstract analytic invariant of $(C,0)$), even if $(C,0)$  is embedded in some $(X,0)$, it cannot be
determined from the embedded topological type of $C\subset X$.

Indeed, in section~\ref{ex:non-rational} we exhibit two topologically equivalent reduced curve germs $C$
and $C_1$ on the non-rational surface $X=\{x^2+y^3+z^7=0\}$ satisfying $\kappa_X(C)=\kappa_X(C_1)=2$,
$\delta(C)=1$, and $\delta(C_1)=0$. This shows that $\kappa_X(C)=\delta(C)$ does not extend to
${\rm Pic}^0(\tilde X)\not=0$ case (even if the link is an integral homology sphere).

These facts motivate to introduce the {\em reduced} and the {\em topological} versions of the 
$\kappa$-invariant by using the corresponding counting functions in the spirit of~\eqref{eq:kapcf}.
In the rest of this section we want to study the analytical and topological properties of these objects and we wish to compare them.

\subsection{A reduced version of the \texorpdfstring{$\kappa$}{kappa}-invariant}
Let $\I_C\subset \V$ be the set of irreducible divisors intersecting $\widetilde C$ and we assume that $\pi$
satisfies \eqref{rescond}, i.e. $E_v$ intersects only one component $\widetilde C_v$ of $\widetilde C$ for each~$v\in \I_C$.
It seems natural to consider the reduced filtration $\mathcal{F}(\ell'_{I_C})$ associated with
$I_C\subset\V$ (cf.~\eqref{eq:redfilt}), since this set $I_C$ carries certain intrinsic geometrical
properties of the curve germ. In this section we give the general definition for the reduced version
of the~$\kappa$-invariant.

As in the previous section we consider a surface singularity with \RHS link and we fix a good
resolution $\pi$. Let $\ell'\in L'$ and $I\subset\V$.

\begin{dfn}
The reduced $\kappa$-invariant of $\ell'$ associated with $\I$
is given by
\begin{equation}\label{kappa_reduced}
\kappa_{X,\I}(\ell'):=\mathfrak{h}_{[Z_K+\ell'], \I}((Z_K+\ell')|_{\I}).
\end{equation}
\end{dfn}
Note that section~\ref{ss:red} implies that $\kappa_{X,\I}$ is given by the counting function associated with
$P_{[Z_K+\ell']}(\tt_{\I})$.

This general definition applies in the presence of a curve $(C,0)$, when $I=I_C$ and $\ell' =\ell'_C$.
However, in general, $\kappa_{X,\I_C}(\ell'_C)$ depends on the choice of a good resolution of $(X,C)$, even when $(C,0)$ is reduced.
Hence, it might be different from $\kappa_X(\ell'_C)$. This phenomenon is illustrated in section~\ref{ex:non-rational}, where
the reduced $\kappa$-invariant of the
curve $C=\{z=x^2+y^3=0\}$ embedded into $X=\{x^2+y^3+z^7=0\}$ has different values computed in two different resolutions.

On the other hand, the main advantage of the reduced $\kappa$-invariant is that it depends on a special subset of vertices
of $\Gamma$, namely, only on those vertices which are associated with exceptional divisors intersecting the irreducible components of $(C,0)$.
This allows one the possibility to compare it directly with certain abstract invariants of the curve.
For instance, in the case of $\kappa_{X,I_C}(\ell'_C)$ for $X$ rational surface singularity, we will prove that
in fact $\kappa_{X,I_C}(\ell'_C)=\kappa_{X}(\ell'_C)$. In particular, this reduced $\kappa$-invariant carries
enough information to determine the full $\kappa$-invariant of an embedded curve, which by Theorem~\ref{thm:mainkappa}
coincides with~$\delta(C)$.

\subsection{Topological counterparts of \texorpdfstring{$\kappa$}{kappa}-invariants}
In~\eqref{eq:kapcf} the $\kappa$-invariant was exhibited as a counting function of the analytic Poincar\'e series of $X$.
It seems natural to define the topological counterpart of this expression obtaining a combinatorial invariant that depends
only on the dual graph $\Gamma$. The extend to which this invariant recovers the original analytic version will be
discussed in this section.

The decorated dual graph $\Gamma$ determines the lattices $L\subset L'$ and the class $[\ell']\in H=L'/L$ of an
element $\ell'\in L'$. Consider the counting function $Q(\ell')$ (see section~\ref{ss:countingfunctions}) of the
topological Poincar\'e series $Z(\tt)$. For any $\ell'\in L'$ and a subset $\I\subset \V$ we set the following:

\begin{dfn}
\begin{equation}\label{def:ktop}
\kappa_{\Gamma}(\ell'):= Q_{[Z_K+\ell']}(Z_K+\ell') \ \, \, \ \mbox{and} \ \, \, \
\kappa_{\Gamma,\I}(\ell'):=Q_{[Z_K+\ell'],\I}(Z_K+\ell').
\end{equation}
\end{dfn}
The example in section~\ref{ex:non-rational} shows that in general
$\kappa_{\Gamma}(\ell')$ and $\kappa_{\Gamma,\I}(\ell')$ can be different and they also might
depend (even under the natural identifications of $\ell'$ and $I$) on the chosen resolution.

\subsection{\texorpdfstring{$\kappa$}{kappa}-invariants for rational singularities}
In order to identify the analytic and topological $\kappa$-invariants
 one needs to use the CDGZ-identity, which is valid for any rational surface germ.
Though the CDGZ-identity is valid even for the larger class of splice quotient singularities,
the equality of the full and reduced $\kappa$-invariants
is specific exactly to rational singularities (it cannot be extended in general for
larger classes). In order to make this fact more transparent, we split the results relating
the $\kappa$-invariants into two: Theorem~\ref{eqCDGZsing} and \ref{thm:eqratsing}.

\begin{thm}\label{eqCDGZsing}
Assume that $(X,0)$ is a \RHS surface singularity and $\pi$ is a good resolution for which the CDGZ-identity holds,
that is, $P(\tt)=Z(\tt)$. Let $\Gamma$ be the dual graph of $\pi$.
Then for any $\ell'\in L'$ and $\I\subset \V$ we have
$\kappa_X(\ell')=\kappa_{\Gamma}(\ell')$ and $\kappa_{X,\I}(\ell')=\kappa_{\Gamma,\I}(\ell')$.
\end{thm}

\begin{proof}
The assumption $P_{[Z_K+\ell']}(\tt)=Z_{[Z_K+\ell']}(\tt)$ for $\pi$ implies that the corresponding coefficient
functions are equal as well. Thus
$$
\kappa_X(\ell') \eqmap{\eqref{eq:kapcf}}
\sum_{\smallmatrix \ell''\ngeq Z_K+\ell' \\ [\ell'']=[Z_K+\ell'] \endsmallmatrix} \mathfrak{p}(\ell'')=Q_{[Z_K+\ell']}(Z_K+\ell')
\eqmap{\eqref{def:ktop}} \kappa_{\Gamma}(\ell').$$
On the other hand, the reduced CDGZ-identity for $\I\subset \V$ is valid too by
section~\ref{ss:red}, which gives the reduced identities $\kappa_{X,\I}(\ell')=\kappa_{\Gamma,\I}(\ell')$.
\end{proof}

Before we state the final result for rational singularities one needs to define the concept
of $E^*$-support of a cycle,
which generalizes the object $I_C$. 
Given a cycle $\ell'\in L'$, we define its support $\Supp^*(\ell')$ with respect to the $\{E^*_v\}$-basis
as $\Supp^*(\ell'):=\{v\in\V \mid (E_v,\ell')\neq 0\}$.
Note that, if $(C,0)$ is a curve germ, then $I_C=\Supp^*(\ell'_C)$.

\begin{thm}\label{thm:eqratsing}
Let $(X,0)$ be a rational singularity. Then for any $\ell'\in\cS'$ one has
\begin{equation}
\label{eq:eqratsing}
\kappa_X(\ell')=\kappa_\Gamma(\ell')=\kappa_{\Gamma,\Supp^*(\ell')}(\ell')=\kappa_{X,\Supp^*(\ell')}(\ell').
\end{equation}
In particular, if $(C,0)$ is a reduced curve germ in $(X,0)$, and $I_C=\Supp^*(\ell'_C)$,
then
$$\kappa_X(C)=\kappa_{X,I_C}(\ell'_C)=\kappa_{\Gamma}(\ell'_C)=\kappa_{\Gamma,I_C}(\ell'_C),$$
which is an embedded topological invariant of~$C\subset X$.
\end{thm}

\begin{proof}
In order to prove~\eqref{eq:eqratsing} we use Theorems~\ref{thm:CDGZ} and \ref{eqCDGZsing}. It is enough to show
$\kappa_X(\ell')=\kappa_{X,\Supp^*(\ell')}(\ell')$, which is equivalent to
$\mathfrak{h}_h(Z_K+\ell')=\mathfrak{h}_{h,I}(Z_K+\ell')$, where $h=[Z_K+\ell']$ and
$I={\rm Supp}^*(\ell')$. Set any $I'$ such that $I\subset I' \subset V$. Then, from the definition of
the relative Hilbert coefficients, there exists an effective integral cycle
$\ell_0=\sum_{v\in V\setminus I'} m_vE_v$, with all $\{m_v\}_{v\in V\setminus I'}$
sufficiently large, such that $\mathfrak{h}_{h,I'}(Z_K+\ell')= \mathfrak{h}_{h}( Z_K+\ell'-\ell_0)$.
Then, by~\eqref{eq:hell} we have
$$\mathfrak{h}_{h,I'}(Z_K+\ell')= \chi(Z_K+\ell'-\ell_0)-h^1(\cO_{{\tilde X}}(-Z_K-\ell'+\ell_0))-\chi(r_h)+h^1(\cO_{{\tilde X}}(-r_h)).$$
Next, consider the cohomology long exact sequence associated with the short exact sequence
of sheaves
$$0\to \cO_{{\tilde X}}(-Z_K-\ell')\to \cO_{{\tilde X}}(-Z_K-\ell'+\ell_0)\to \cO_{\ell_0}(-Z_K-\ell'+\ell_0)\to 0.$$
Since $h^1(\cO_{{\tilde X}}(-Z_K-\ell'))=0$ by Grauert-Riemenschneider's vanishing Theorem, we obtain that
$ h^1(\cO_{{\tilde X}}(-Z_K-\ell'+\ell_0))=h^1(\cO_{\ell_0}(-Z_K-\ell'+\ell_0))$.
Now, $(\ell', E_v)=0$
for any $v\in V\setminus I'$, hence the line bundle $\cO_{\ell_0}(\ell')$ is numerically trivial.
Since $\ell_0$ is supported by the exceptional set of a resolution of a rational singularity, this 
bundle is in fact trivial. This and Serre duality imply
$$h^1(\cO_{\ell_0}(-Z_K-\ell'+\ell_0))=h^1(\cO_{\ell_0}(-Z_K+\ell_0))=h^0(\cO_{\ell_0}).$$
Since $X$ is rational, $h^1(\cO_{\ell_0})=0$, hence $h^0(\cO_{\ell_0})=\chi(\ell_0)$.
In conclusion (using $ \chi(Z_K+\ell'-\ell_0)=\chi(\ell_0-\ell')$ and $(\ell_0,\ell')=0$) we obtain
\begin{equation*}\label{eq:frakH}\begin{split}
\mathfrak{h}_{h,I'}(Z_K+\ell')&=\chi(\ell_0-\ell')-\chi(\ell_0)-\chi(r_h)+h^1(\cO_{{\tilde X}}(-r_h))\\
&=\chi(-\ell')-\chi(r_h)+h^1(\cO_{{\tilde X}}(-r_h)).\end{split}\end{equation*}
This applied for $I'=I$ and $I'=V$ gives the statement.
\end{proof}

In fact, based on the above proof, we can also conclude in the next corollary that for $(X,0)$ rational the surgery formula~\eqref{eq:surgform},
as well as the quasi-polynomiality of the counting function (Corollary~\ref{cor:sharp}\ref{cor:sharp2}), extends to the entire $Z_K+{\mathcal S}'$.

\begin{cor}
\label{cor:surgform}
If $X$ is rational, then $Q^{\Gamma}_h(Z_K+\ell')=Q^{\Gamma}_{h,I}(Z_K+\ell')$ for
$\ell'\in {\mathcal S}'$, $h=[Z_K+\ell']$, and $I=\Supp^*(\ell')$.

Moreover, in the case of the surgery formula~\eqref{eq:surgform} one has $Q^{\Gamma_k}_{[j^*_k(Z_K+\ell')]}(j^*_k(Z_K+\ell'))=0$. In particular,
 ~\eqref{eq:surgform} holds for ~$Z_K+{\mathcal S}'$.
\end{cor}

\begin{proof}
In the proof above one shows $\mathfrak{h}_h(Z_K+\ell')=\mathfrak{h}_{h,I}(Z_K+\ell')$.
By Theorem~\ref{thm:CDGZ}, this implies $Q^{\Gamma}_h(Z_K+\ell')=Q^{\Gamma}_{h,I}(Z_K+\ell')$.
It remains to verify the vanishing of the correction term
$Q^{\Gamma_k}_{[j^*_k(Z_K+\ell')]}(j^*_k(Z_K+\ell'))$.

Note that $j^*_k(\ell')=0$ by the choice of the set $I$ and the definition of the dual operators~$j_k^*$.
Thus, using~\eqref{eq:proj\ZK} one has
$Q^{\Gamma_k}_{[j^*_k(Z_K+\ell')]}(j^*_k(Z_K+\ell'))=Q^{\Gamma_k}_{[Z_K(\Gamma_k)]}(Z_K(\Gamma_k))$.
Moreover, the duality~\cite[Theorem 4.4.1]{LNNdual} (see also Corollary~\ref{olddual}) for counting
functions gives the identity
$$Q^{\Gamma_k}_{[Z_K(\Gamma_k)]}(Z_K(\Gamma_k))=\pc^{\cS'_{\mathbb{R}}(\Gamma_k)}(Z^{\Gamma_k}_0(\tt)),$$
where $\cS'_{\mathbb{R}}(\Gamma_k)$ is the real Lipman cone and $Z^{\Gamma_k}_0(\tt)$
is the $0$-part of the topological Poincar\'e series associated with $\Gamma_k$. Then
$\Gamma_k$ (being a subgraph of a rational graph) is rational, for all $k$, and hence
$Z^{\Gamma_k}_0(\tt)=P^{\Gamma_k}_{0}(\tt)$ and
$\pc^{\cS'_{\mathbb{R}}(\Gamma_k)}(Z^{\Gamma_k}_0(\tt))=p_g(X_k)=0$.
\end{proof}

\section{A formula for $\delta$ via periodic constants}
\label{sec:formula-delta}

\subsection{The formula} Let $C=\bigcup_{i\in \I}C_i$ be the irreducible decomposition (with finite set $\I$) of a reduced
curve germ and consider $C_J=\bigcup_{j\in J} C_j$ for any $\emptyset\neq J\subset \I$. As in
section~\ref{sec:relabsP} one defines the Hilbert series
$H_{C_J}(\tt_J):=\sum_{\ell_J}\mathfrak{h}_{C_J}(\ell_J)\tt^{\ell_J}_J$, and following~\eqref{eq:coefC}
we also consider the Poincar\'e series
$P_{C_J}(\tt_J):=\sum_{\ell_J}\mathfrak{p}_{C_J}(\ell_J)\tt^{\ell_J}_J$ of any such~$C_J$.

 The aim of this section is to prove the following formulae for the delta invariant of $(C,0)$ in terms of the 
 periodic constants associated with these objects.

\begin{thm}\label{thm:delta-sec}
\mbox{}
\begin{enumerate}[leftmargin=3cm,itemindent=.5cm,labelwidth=\itemindent,labelsep=.5cm,align=left,label={$(\alph*)$}]
 \item For any $i\in I$ one has $\delta(C_i)= -{\rm pc}(P_{C_i}(t_i))$, and
 \item \label{eq:deltapc}
$\delta(C)=
\sum_{i\in I}\delta(C_i) +
\sum_{J\subset \I, \ |J|>1} (-1)^{|J|} \ P_{C_{J}}(1_{J}).$
\end{enumerate}
\end{thm}
Note that even though the proof of Theorem~\ref{thm:delta-sec} will be heavily based on the concept of periodic constant,
in the recursive identity~\ref{eq:deltapc} its presence is completely hidden, even missing.

We would like to emphasize that the delta invariant can be determined recursively from the delta invariant of
the components and the \textit{Hironaka generalized intersection multiplicity} as well. Namely,
if we consider two (not necessarily irreducible) germs $(C_1',0)$ and $(C_2',0)$ without common irreducible components, 
embedded in some $(\CC^n,0)$ such that $(C_i',0)$ is defined by the ideal $I_i$ in $\cO_{(\CC^n,0)}$ ($i=1,2$),
then \textit{Hironaka's intersection multiplicity} is defined by $(C_1',C_2')_{Hir}:=\dim _{\CC}\, (\cO_{(\CC^n,0)}/ I_1+I_2)$. 
Then, one has the following \textit{Hironaka's formula}~\cite{Hironaka} (see also~\cite[Sect.~5]{kappa} for a detailed discussion)
\begin{equation}\label{eq:delta12}
\delta(C)=\sum_{i\in I} \delta(C_i)+ \sum_{i=1}^{r-1} (C_i,C^i)_{\text{Hir}},
\end{equation}
where $C^i:=\cup_{j=i+1}^rC_j$.

Formally~\eqref{eq:deltapc} has some similarity with~\eqref{eq:delta12} but their equivalence is not apparent.
Nevertheless, our proof will be independent of~\eqref{eq:delta12}.

Before we present the proof, in the next section we discuss the formula for two key families:
plane curves and ordinary $r$-tuples of curves in $\CC^r$. In a way,
they constitute the two {\it opposite} extreme cases. In our discussion we will stress
the major differences between them.

\subsection{Special cases}\label{ss:spec}
\begin{exam}
Assume $(C,0)\subset (\CC^2,0)$ is an irreducible plane curve singularity.
Then the identity $\delta(C)=-\pc(P_{C}(t))$ is well known, but in a different form,
see eg.~\cite{CDGZ-Moncurve} and~\cite[Section 7.1]{LSzMonoids}. Indeed, by~\cite{CDGZ-Moncurve}
the Poincar\'e series $P_{C}(t)$ equals the generating series of the numerical semigroup associated with $(C,0)$,
which implies the statement by Example~\ref{ex:pc}\eqref{ex:pc3}: the negative of the periodic constant equals
the number of gaps, hence the delta invariant of the curve.
\end{exam}
\begin{exam}
Assume that $C=\cup_{i\in I} C_i$ ($I=\{1,\dots, r\}$) is a plane curve singularity. We recall two basic
properties, both of them specific to the case of plane curves.

First, in this context there exists the notion of Alexander invariant associated with
the covering $\CC^2\setminus C$ (since this space has non-trivial fundamental group),
hence one can define the (multivariable) Alexander polynomial $\Delta(t_1,\ldots, t_r)$ of $C$.
Then, see eg.~\cite{cdg,cdg2}, $P_{C_i}(t_i)=\Delta_{C_i}(t_i)/(1-t_i)$ and
$P_C(t_1, \ldots t_r)=\Delta(t_1,\dots, t_r)$ whenever $r\geq 2$.

Second, there is an inductive Torres formula of type~\cite{torres}
 \begin{equation}\label{eq:redP}
P_{C_{I\setminus \{1\}}}(t_2,\ldots,t_r)=
P_C(t_1,\ldots,t_r)|_{t_1=1}\cdot (1-t_2^{(C_1,C_2)}\cdots t_r^{(C_1,C_r)})^{-1}.
\end{equation}
Here $(C_i,C_j)$ denotes the usual intersection multiplicity (which in the case of plane curves
coincides with Hironaka's definition). Since $\Delta_{C_i}(1)=1$, Torres' formula applied to
$C_{\{{i,j}\}}$ proves $P_{{C_{\{i,j\}}}}(1,1)=(C_i,C_j) $ ($i\not=j$), and, again by
Torres formula applied for $|J|\geq 3$ we get that $P_{C_J}(1_J)=0$ if $|J|\geq 3$.
In particular,~\eqref{eq:deltapc} gives the well-known formula, valid for plane curves,
\begin{equation}\label{eq:deltaplanecurves}
\delta(C)=\sum_i\delta(C_i)+\sum_{i<j} (C_i,C_j).
\end{equation}
This can be deduced from Hironaka's formula~\eqref{eq:delta12} as well, since for plane curves
$(C',C_1\cup C_2)=(C',C_1)+(C',C_2)$, a property which usually
fails for non-plane germs in the general context of Hironaka's intersection multiplicity.

For certain properties in the Gorenstein case see e.g. \cite{Moy}.
\end{exam}

\begin{exam}
Let $(C,0)$ be analytically equivalent to a union of $r$ coordinate axes in $\CC^r$, that is, $(C,0)$ is an ordinary
$r$-tuple, then a straightforward computation shows $P_{C_J}(1_J)=1$ if $|J|\geq 2$. Since $\delta(C_i)=0$
for all $i$, Theorem~\ref{thm:delta-sec}\ref{eq:deltapc} reads as $\delta(C)=\sum_{j=2}^r (-1)^j\binom{r}{j}=r-1$.
\end{exam}

\subsection{Proof of Theorem~\ref{thm:delta-sec}}

The proof of Theorem~\ref{thm:delta-sec} is based on the
 inversion formula of Gorsky and N\'emethi~\cite{GN}, and Moyano-Fern\'andez~\cite{Moy},
which recovers the Hilbert function $\mathfrak{h}_C$ from the Poincar\'e series $P_{C_J}$, thus \textit{inverts}
formula~\eqref{eq:coefC} (and by its \textit{formal} proof, it is valid for any curve germ $(C,0)$ as well).

Note that by definition (see section~\ref{sec:relabsP}) if $\ell_i=0$ for all $i\notin J$,
then $\mathfrak{h}_{C_J}(\ell_J)=\mathfrak{h}_C(\ell_I)$, i.e. $H_{C_J}(\tt_J)=H_{C}(\tt_I)|_{t_i=1,i\notin J}$.
It also makes sense to write $\mathfrak{h}_C(\ell_J)$ so that $\ell_J$ is extended to a vector $\ell_I$ by setting
its entries indexed by $I\setminus J$ to be zero.

Then, the Hilbert function satisfies the following properties:
\begin{itemize}
 \item[(a)] $\mathfrak{h}_C(\ell_I)=\mathfrak{h}_C(\mathrm{max}\{\ell_I,0\})$;
 \item[(b)] $\mathfrak{h}_C(0)=0$;
 \item[(c)] for any $J\subset I$ we have
 $\mathfrak{p}_{C_J}(\ell_J)=\sum_{I'\subset I}(-1)^{|I'|-1} \mathfrak{h}_{C}(\ell_J+1_{I'})$.
\end{itemize}
Therefore~\cite[Theorem 3.4.3]{GN} applies and it implies the following inversion formula:
\begin{equation}\label{eq:inv}
\mathfrak{h}_C(\ell_I)=\sum_{J\subset I} (-1)^{|J|-1}
\sum_{0\leq \tilde \ell \leq \ell_J-1_J}\mathfrak{p}_{C_J}(\tilde \ell).
\end{equation}

\begin{proof}[Proof of Theorem~\ref{thm:delta-sec}]
By Lemma~\ref{lem:Ppoly}, $P_{C_J}(\tt_J)$ is a polynomial for $|J|>1$, hence for \textit{big enough} $\ell_J$,
the constant sum $\sum_{0\leq \tilde \ell \leq \ell_J-1_J}\mathfrak{p}_{C_J}(\tilde \ell)$
is the periodic constant and equals $P_{C_J}(1_J)$, cf. Example~\ref{ex:pc}~\eqref{ex:pc1}.

If $J=\{i\}$ for some $i\in I$, then the counting function is
$\sum_{0\leq \tilde \ell \leq \ell_i-1}\mathfrak{p}_{C_i}(\tilde \ell)$ whose constant term is the periodic
constant for $\ell_i\gg 0$. On the other hand, by~\eqref{eq:inv} and~\eqref{eq:sgpid} for $\ell\geq c$,
where $c$ is the conductor, we get
$$\delta(C)=|\ell|+\sum_{\smallmatrix J\subset I \\ |J|>1 \endsmallmatrix} (-1)^{|J|} P_{C_{J}}(1_{J})-\sum_{i\in I}
\sum_{0\leq \tilde \ell \leq \ell_i-1}\mathfrak{p}_{C_{i}}(\tilde \ell).$$
Since the left hand side of the above identity is constant the result follows.
\end{proof}

\section{Twisted duality for the topological Poincar\'e series}
\label{sec:twisted-duality}

\subsection{Ehrhart-MacDonald-Stanley duality for rational functions}

We will use~\cite{LNNdual} as a general reference for this section, where more details can be found.
\subsubsection{}
We fix a free $\ZZ$-module $L$ and an overlattice $L'\supset L$ of the same rank and denote by $H$ the finite quotient
group $L'/L$ of order, say $d$. Let us fix a basis $\{E_v\}_{v\in\V}$ in $L$.
In the sequel we will adapt the notation of sections~\ref{sec:lattice} and~\ref{ss:set} to this more general context.

We consider the following type of multivariable rational functions (in the variables $\tt^{L'}$) with rational exponents
in $\left(\frac{1}{d}\ZZ\right)^{|\V|}$:
\begin{equation}\label{eq:func}
z(\tt)=\frac{\sum_{k=1}^r\iota_k\tt^{b_k}}{\prod_{i=1}^n
(1-\tt^{a_i})},
\end{equation}
where $\iota_k \in\mathbb{Z}$, $b_k\in L'$, and $a_i=\sum_{v\in\V}a_{i,v}E_v\in L'$, such that $a_{i,v}$ are all strictly
positive.

Let $Tz(\tt)=\sum_{\ell'}z(\ell')\tt^{\ell'}\in \ZZ[[\tt^{1/d}]][\tt^{-1/d}]$ be the formal Taylor expansion of $z(\tt)$
at the origin and we also write the Taylor expansion of $z(\tt)$ at infinity in the form
$$T^{\infty}z(\tt)=\sum_{\tilde{\ell}}z^{\infty}(\tilde{\ell})\tt^{\tilde{\ell}}\in \ZZ[[\tt^{-1/d}]][\tt^{1/d}],$$
where $T^\infty z$ is obtained by
the substitution $\mathbf{s}=1/\tt$ into the Taylor expansion at $\mathbf{s}=0$ of the function $z(1/\mathbf{s})$.

The function $z$ has a decomposition $\sum_{h\in H}z_h(\tt)$ with respect to $H=L'/L$, where $z_h(\tt)$
is rational of the form $\sum_{b'\in h+L}\iota_{b'}\tt^{b'}/\prod_{i=1}^n (1-\tt^{da_i})$ ($\iota_{b'}\neq 0$ for finite $b'$).
The decompositions $\sum_h (Tz)_h$ and $\sum_h (T^{\infty}z)_h$ of the series $Tz$ and $T^{\infty}z$ are defined similarly as  in section~\ref{ss:set}.

Once $h\in H$ fixed, for any subset $\I\subset \V$ we consider the reduced functions $z_h(\tt_{\I})$
and their series $(Tz)_h(\tt_{\I})$ and $(T^{\infty}z)_h(\tt_{\I})$ by substituting $t_i=1$ for all $i\notin \I$
(cf. section~\ref{ss:set}).

\subsubsection{\bf Quasi-polynomials and duality}
Let $q^{(z)}_{h,\I}$ be the modified counting function associated with $(Tz)_h(\tt_{\I})$ as defined in~\eqref{eq:modcount1}.
Following the results of~\cite{LNN,LSzPoincare}, by the Ehrhart theory of polytopes associated with the denominator of $z$
one can consider the following chamber decomposition of ${\rm pr}_\I(L'\otimes {\mathbb R})={\mathbb R}^{|\I|}$:
let $\mathcal{B}_\I$ be the set of all bases $\sigma$ of ${\mathbb R}^{|\I|}$ such that $\sigma$ is a subset of the given 
configuration of vectors 
$\{a_i|_\I\}_{i=1,\dots,n}\cup \{E_v|_I\}_{v\in\I}$ in ${\mathbb R}^{|\I|}$ 
(here we use notation $a_i|_\I:={\rm pr}_\I(a_i)$ as in section \ref{ss:set}).
Then a (big, open) chamber $\mathfrak{c}$ is a connected component of
${\mathbb R}^{|\I|}\setminus \cup_{\sigma\in\mathcal{B}_\I}\partial\R_{\geq 0}\sigma$, where
$\partial\R_{\geq 0}\sigma$ is the boundary of the closed cone~$\R_{\geq 0}\sigma$.

Then, by~\cite[Corollary 3.5.1]{LNNdual} it is known that for any fixed chamber $\mathfrak{c}$ the
modified counting function $q^{(z)}_{h,\I}$ admits a quasi-polynomial in the sense of section~\ref{ss:pc}, namely
$\qp{q}^{\, \mathfrak{c},(z)}_{h,\I}$ satisfying
$\qp{q}^{\, \mathfrak{c},(z)}_{h,\I}(\ell)=q^{(z)}_{h,\I}(r_h+\ell)$
for any $\ell$ \textit{sitting sufficiently deeply in the interior of $\mathfrak{c}$}. Thus, $(Tz)_h$
(often referred to as $z_h$) admits a periodic constant
$\mathrm{mpc}^{\mathfrak{c}}(z_h(\tt_{\I}))=\qp{q}^{\, \mathfrak{c},(z)}_{h,\I}(0)$
associated with the chamber $\mathfrak{c}$. Moreover, the next duality result from~\cite{LNNdual} shows that (under some
conditions) the modified periodic constant associated with a chamber $\mathfrak{c}$ coincides with a finite sum of certain
coefficients of the Taylor expansion at infinity.

\begin{thm}[{\cite[Th.~3.6.1]{LNNdual}}]\label{thm:mpc}
Fix $h$ and $\I$ as above. Denote by
$(T^{\infty}z)_h(\tt_\I)=\sum_{\ell'}z^{\infty}_{h,\I}(\ell')\tt_\I^{\ell'}$
the $h$-component of the Taylor expansion of $z(\tt)$ at infinity.
Assume that for a fixed chamber $\mathfrak{c}$ one has
$b_k|_\I\in\mathfrak{c}$ for all $k$.
Then
\begin{equation}\label{eq:mpc}
\mathrm{mpc}^{\mathfrak{c}}(z_h (\tt_\I))=\sum_{\ell'\geq 0}z^{\infty}_{h,\I}(\ell').
\end{equation}
\end{thm}
Note that the chamber decomposition depends essentially on the set of \textit{poles} $\{a_i\}_{i=1}^n$
(see the definition of $\mathcal{B}_\I$). The construction of the chambers shows that
by decreasing the set of poles one might enlarge certain chambers, which can help improve results on
periodic constant calculations (cf.~Theorem~\ref{thm:mpc} and Proposition~\ref{prop:LSz}).

\subsection{\bf The twisted zeta function}
We now return to the \RHS surface singularity case, where the dual graph $\Gamma$ is fixed. Consider the zeta function
$f(\tt)$ as defined in~\eqref{eq:1.1} and the topological Poincar\'e series $Z(\tt)$ as its Taylor expansion at the origin.
We recall some important results from~\cite{LSzPoincare}, which will be used in the sequel.

For a subset $\emptyset\neq\I\subset \V$, we define its closure $\overline{\I}$ as the set of vertices of the
\textit{connected} minimal full subgraph $\Gamma_{\overline{\I}}$ of $\Gamma$, which contains $\I$.
Note that this graph is unique since $\Gamma$ is a tree.
Denote by $\text{val}_{\overline{\I}}(v)$ the valence of a vertex $v\in \overline{\I}$ in the graph $\Gamma_{\overline{\I}}$.

Then, in~\cite[Lemma 7]{LSzPoincare} it was proved that $f(\tt_\I)$ has a product decomposition of type
\begin{equation}\label{eq:fI}
f(\mathbf{t}_\I) = {\rm Pol}(\mathbf{t}_{\I}) \cdot \prod_{v\in \overline{\I}}
 \Big( 1 -\mathbf{t}_{\I}^{E^{*}_{v}} \Big)^{\text{val}_{\overline{\I}}(v)-2}=
 {\rm Pol}(\mathbf{t}_{\I}) \cdot {\rm Prod}(\tt_I),
\end{equation}
where ${\rm Pol}(\mathbf{t}_{\I})$ is a finite sum supported on $\pi_\I(\cS')$, in particular it has no poles.
Hence, the possible set of poles of $f$ via the $\I$-reduction is reduced from the set of zeros of
$\prod_{v\in \cE} (1-\tt_\I^{E^*_v})$ to the set of zeros of
$\prod_{v\in \cE_{\overline{\I}}} (1-\tt_\I^{E^*_v})$,
where $\cE$ (resp. $\cE_{\overline{\I}}$) is the set of end-vertices of $\Gamma$ (resp. $\Gamma_{\overline{\I}}$).
Note that $\cE_{\overline{\I}}\subset \I$. Therefore, by the construction of the chamber decomposition of
${\mathbb R}^{|\I|}$ the chambers associated with $f(\tt_\I)$ can be
determined by the bases selected from the vector configuration
$\{E^{*}_{v}|_\I\}_{v\in \mathcal{E}_{\overline{\I}}}\cup \{E_u|_\I\}_{u\in \I}$.
Moreover, one can prove the following:
\begin{prop}[\cite{LSzPoincare}]
\label{prop:LSz}
For any $\I\subset\V$ the interior of the projected Lipman cone $\textnormal{int}\, (\pi_{\I}(\mathcal{S}'_{\mathbb{R}}))$
is contained entirely in a (big) chamber $\mathfrak{c}$ of $f(\tt_\I)$. Thus, the modified counting function
$q^{(Z)}_{h,\I}$ associated with $Z_h(\tt_{\I})$ admits a quasi-polynomial in $\pi_{\I}(\mathcal{S}'_{\mathbb{R}})$,
which will be denoted simply by $\qp{q}_{h,\I}$ in the sequel.
\end{prop}

Now, we consider the \textit{twisted zeta function} $f_{\ell'_0}(\tt):=\tt^{\ell'_0}\cdot f(\tt)$ for some fixed
$\ell'_0\in\cS'$ (see section~\ref{sec:tw-series}).
By~\eqref{eq:fI} the chamber decomposition associated with $f_{\ell'_0}(\tt_\I)$ is the same as the one of $f(\tt_\I)$.
Moreover, if $f(\tt_I)$ has product factorization
${\rm Pol}(\tt_I)\cdot {\rm Prod}(\tt_I)$,
then $f_{\ell'_0}(\tt_\I)={\rm Pol}_{\ell'_0}(\tt_I)\cdot {\rm Prod}(\tt_I)$, where
${\rm Pol}_{\ell'_0}(\tt_\I)=\tt^{\ell'_0}_{\I}\cdot {\rm Pol}(\tt_\I)$ is still a finite sum
and it is still supported on $\pi_\I(\cS')$ since $\ell'_0\in\cS'$.

\begin{rem}\label{cor:frel}
The discussion above allows for the extension of the previous results to the twisted case. In particular,
Theorem~\ref{thm:mpc} and Proposition~\ref{prop:LSz} also hold for the twisted zeta
function~$f_{\ell'_0}(\tt)=\tt^{\ell'_0}\cdot f(\tt)$ as well.
\end{rem}

\subsection{Periodic constants for twisted functions}\label{sec:tw-series}
Let us fix a cycle $\ell'_0\in L'$ with $[\ell'_0]=h_0$ as before.
Motivated by the previous section, it will be useful to compare invariants of a series with its \textit{ twisted series}: for any fixed series
$S(\tt)$ we set the twisted
$R(\tt):=\tt^{\ell'_0}S(\tt)$. By a straightforward calculation
\begin{equation}
\array{rcl}
\label{eq:tw-series}
R_{h+h_0}(\tt)&=&\tt^{\ell'_0}S_h(\tt)\\[0.1cm]
Q_{h+h_0}^{(R)}(x)&=&Q_h^{(S)}(x-\ell'_0).
\endarray
\end{equation}
 For any $h\in H$ we define the
\textit{dual shift $\check\ell_{0}(h)\in L'_h$ of $\ell'_0$ with respect to the class $h$} by
\begin{equation}\label{eq:dualrep}
\check\ell_{0}(h):=\ell'_0 + r_{h-h_0}.
\end{equation}
For $h=h_0$  the definition gives the cycle itself: $\check\ell_0(h_0)=\ell'_0$.
We will use the simplified notation $\check\ell_{0}$ for the dual shift $\check\ell_0(0)$.
In this case one obtains $\check\ell_{0}\in L$ satisfying $\check\ell_0=\ell'_{0}+r_{-h_0}$
(or $\check\ell_0=\lceil \ell_0'\rceil$). This rewritten in the form
$\ell'_{0}=\check\ell_0- r_{-h_0}$ can be compared with
the usual decomposition $\ell'_0=\ell_0+r_h$ for $\ell_0=\lfloor \ell'_0\rfloor\in L$.
This \textit{symmetry} (and its application in the duality Theorem~\ref{tdual})
explain the term \textit{dual}. Note that, if $h_0=0$, then $\check{\ell}_{0}=\ell_0$.

The following relations for the quasi-polynomials
associated with the (modified) counting functions are straightforward from~\eqref{eq:tw-series}:
\begin{equation}
\label{eq:tw-series2}
\qp{Q}_{h+h_0}^{(R)}(\ell)=\qp{Q}_h^{(S)}(\ell+r_{h}-\check{\ell}_0(h)), \ \ \
\qp{q}_{h+h_0}^{\, (R)}(\ell)=\qp{q}_h^{\, (S)}(\ell+r_{h}-\check{\ell}_0(h)).
\end{equation}
Note that the evaluation of the quasi-polynomial $\qp{Q}_h^{(S)}(\ell)$ at zero provides the periodic constant
of $S_h(\tt)$, and evaluation of $\qp{Q}_{h+h_0}^{(R)}(\ell)$ at zero
is the periodic constant of $\tt^{\ell'_0}S_h(\tt)$:
\begin{equation}
\label{eq:pc}
\pc^\cK(\tt^{\ell'_0}S_h(\tt))=\qp{Q}_h^{(S)}(r_{h}-\check{\ell}_0(h)),\ \ \
\mpc^\cK(\tt^{\ell'_0}S_h(\tt))=\qp{q}_h^{\, (S)}(r_{h}-\check{\ell}_0(h)).
\end{equation}
Analogously, in the \rel situation, if $\I\subset\V$, then
\begin{equation}
\label{eq:pcred}
\pc^\cK(\tt_\I^{\ell'_0}S_{h}(\tt_\I))=\qp{Q}_{h,\I}^{(S)}(r_{h}-\check{\ell}_0(h)),\ \ \
\mpc^\cK(\tt_\I^{\ell'_0}S_{h}(\tt_\I))=\qp{q}_{h,\I}^{\, (S)}(r_{h}-\check{\ell}_0(h)).
\end{equation}

Again, as mentioned after~\eqref{eq:InEx}, the \rel quasi-polynomials
$\qp{Q}_{h,\I}^{(S)}$ and $\qp{q}_{h,\I}^{\, (S)}$ can formally be applied to elements in $L$ via the
projection $L\to L_\I$.

\subsection{Twisted duality for counting functions of Poincar\'e series}\label{ss:td}
Recall that for any $\emptyset\neq\I\subset \V$, for simplicity we have denoted by $Q_{h,\I}(\ell')$
(resp. $q_{h,\I}(\ell')$) the counting function (resp. modified counting function) associated with
$Z_h(\mathbf{t}_{\I})$ for any $h\in H$. They admit quasi-polynomials $\qp{Q}_{h,\I}(\ell)$
(resp. $\qp{q}_{h,\I}(\ell)$) associated with the cone $\pi_{\I}(\cS'_{\RR})$ so that
for $\ell'=r_h+\ell\in \cS'$ with $\ell\gg 0$ in $L$ one has $Q_{h,\I}(\ell')=\qp{Q}_{h,\I}(\ell)$
(resp. $q_{h,\I}(\ell')=\qp{q}_{h,\I}(\ell)$). In particular,
$\mathrm{pc}^{\pi_{\I}(\cS'_{\mathbb{R}})}(Z_h(\tt_{\I}))=\qp{Q}_{h,\I}(0)$
(resp. $\mathrm{mpc}^{\pi_{\I}(\cS'_{\mathbb{R}})}(Z_h(\tt_{\I}))=\qp{q}_{h,\I}(0)$) by definition.

The aim of this section is to prove the following \textit{twisted duality theorem} for the counting functions of
the topological Poincar\'e series.

\begin{thm}
\label{tdual}
For any fixed $\emptyset\neq\I\subset\V$, $\ell_0'\in\cS'$ with $[\ell'_0]=h_0$ and $h\in H$ the following identities hold:
\begin{enumerate}[leftmargin=3cm,itemindent=.5cm,labelwidth=\itemindent,labelsep=.5cm,align=left,label={$(\alph*)$}]
 \item\label{tdual:1}
 $\mathrm{mpc}^{\pi_{\I}(\cS'_{\mathbb{R}})}((f_{\ell'_0})_h(\tt_{\I}))=\qp{q}_{h-h_0,\I}(r_h-\check{\ell}_{0}(h))=
 q_{[\ZK]-h+h_0,\I}(\ZK-r_h+\ell'_0)$,
 \item\label{tdual:2}
 $\mathrm{pc}^{\pi_{\I}(\cS'_{\mathbb{R}})}((f_{\ell'_0})_h(\tt_{\I}))=\qp{Q}_{h-h_0,\I}(r_h-\check{\ell}_{0}(h))=
 Q_{[\ZK]-h+h_0,\I}(\ZK-r_h+\ell'_0)$,
\end{enumerate}
where $\check{\ell}_{0}(h)$ is the dual shift of $\ell'_0$ by $h$ as defined in~\eqref{eq:dualrep}.
\end{thm}

\begin{proof}
Part~\ref{tdual:1} implies~\ref{tdual:2} by the inclusion-exclusion principle~\eqref{eq:InEx},
hence we only have to show~\ref{tdual:1}.

Consider the twisted zeta function $f_{\ell'_0}(\tt)=\tt^{\ell'_0}\cdot f(\tt)$.
Then for any $\emptyset\neq\I\subset\V$ and $h\in H$ the Taylor expansion at the origin can be written as
$(Tf_{\ell'_0})_h(\tt_{\I})=\tt_{\I}^{\ell'_0}\cdot Z_{h-h_0}(\tt_{\I})$.
If $q^{(f_{\ell'_0})}_{h,\I}(\ell')$ denotes the modified counting function associated with
$(Tf_{\ell'_0})_h(\tt_{\I})$, then by~\eqref{eq:tw-series} one obtains
$q^{(f_{\ell'_0})}_{h,\I}(\ell')=q_{h-h_0,\I}(\ell'-\ell'_0)$.
By Remark~\ref{cor:frel}, $q^{(f_{\ell'_0})}_{h,\I}$ as well as $q_{h-h_0,\I}$ admit a quasi-polynomial
associated with $\pi_{\I}(\cS'_{\mathbb{R}})$.
Using~\eqref{eq:tw-series2} one obtains
$\qp{q}^{(f_{\ell'_0})}_{h,\I}(\ell)=\qp{q}_{h-h_0,\I}(\ell+r_h-\check{\ell}_{0}(h))$.
Note that both quasi-polynomials are associated with the special cone $\pi_{\I}(\cS'_{\mathbb{R}})$.
In particular, by~\eqref{eq:pcred} one has
\begin{equation}\label{eq:mpc1}
\mathrm{mpc}^{\pi_{\I}(\cS'_{\mathbb{R}})}((f_{\ell'_0})_h(\tt_{\I}))=
\qp{q}^{(f_{\ell'_0})}_{h,\I}(0)=\qp{q}_{h-h_0,\I}(r_h-\check{\ell}_{0}(h)).
\end{equation}

All the exponents in the numerator of $f_{\ell'_0}(\tt_{\I})$ are situated in
$\pi_{\I}(\cS'_{\RR})$ since the same holds for $f(\tt_{\I})$ and $\ell'_0\in \cS'$.
Therefore, Theorem~\ref{thm:mpc} can be applied to $f_{\ell'_0}(\tt_{\I})$ such that if the $h$-component
of the Taylor expansion at infinity of $f_{\ell'_0}(\tt_{\I})$ is written as
$(T^{\infty}f_{\ell'_0})_h(\tt_{\I})=
\sum_{\tilde{\ell}}(f_{\ell'_0}^{\infty})_{h,\I}(\tilde{\ell})\tt_{\I}^{\tilde{\ell}}$
then
\begin{equation}\label{eq:mpc2}
\mathrm{mpc}^{\pi_{\I}(\cS'_{\mathbb{R}})}((f_{\ell'_0})_h(\tt_{\I}))=
\sum_{\tilde{\ell}\geq 0}(f_{\ell'_0}^{\infty})_{h,\I}(\tilde{\ell}).
\end{equation}
On the other hand, using our previous notation  
$Tf(\tt)=Z(\tt)=\sum_{\ell'\in\cS'}z(\ell')\tt^{\ell'}$ (cf. section \ref{ss:PZ}) we can write
$$(T^{\infty}f_{\ell'_0})(\tt_{\I})=\tt_{\I}^{\ell'_0}(T^{\infty}f)(\tt_{\I})=
\tt_{\I}^{\ell'_0}\sum_{\ell'\in \cS'}z(\ell')\tt_{\I}^{\ZK-E-\ell'},$$
where the second identity follows by the symmetry $f(\tt_{\I})=\tt_{\I}^{\ZK-E} \cdot f(\tt_{\I}^{-1})$.
Thus, by~\eqref{eq:mpc1} and~\eqref{eq:mpc2} one obtains
$$\mathrm{mpc}^{\pi_{\I}(\cS'_{\mathbb{R}})}((f_{\ell'_0})_h(\tt_{\I}))=\qp{q}_{h-h_0,\I}(r_h-\check{\ell}_{0}(h))
=\sum_{\smallmatrix \ell'_{\I}\leq (\ZK-E+\ell'_0)_{\I},\\ [\ell']=[\ZK]-h+h_0\endsmallmatrix}z(\ell').$$
Note that since the sum considers only $\ell'\in L'$ with $[\ell']=[\ZK]-h+h_0$, the condition
$\ell'_{\I}\leq (\ZK-E+\ell'_0)|_{\I}$ is equivalent to $\ell'_{\I}\prec (\ZK-r_h+\ell'_0)|_{\I}$.
Hence, the sum coincides with the modified counting function $q_{[\ZK]-h+h_0,\I}(\ZK-r_h+\ell'_0)$,
which proves~\eqref{tdual:1}.
\end{proof}

The above twisted duality has two important specializations. The first one gives back the already known duality
result proved in~\cite{LNNdual}, which was the motivation for the twisted version as well.
\begin{cor}[{\cite[Theorem 4.4.1]{LNNdual}}]\label{olddual}
If $(X,0)$ is a \RHS surface singularity, then
\begin{enumerate}[leftmargin=3cm,itemindent=.5cm,labelwidth=\itemindent,labelsep=.5cm,align=left,label={$(\alph*)$}]
 \item 
$\mathrm{mpc}^{\pi_{\I}(\cS'_{\mathbb{R}})}(Z_h(\tt_{\I}))=q_{[\ZK]-h,\I}(\ZK-r_h);$
 \item 
$\mathrm{pc}^{\pi_{\I}(\cS'_{\mathbb{R}})}(Z_h(\tt_{\I}))=Q_{[\ZK]-h,\I}(\ZK-r_h).$
\end{enumerate}
\end{cor}

\begin{proof}
Use Theorem~\ref{tdual} for $\ell'_0=0$.
\end{proof}
The second specialization can be considered as a relative version of the duality and it will be applied in the
sequel to the case of an embedded reduced curve germ.
\begin{cor}\label{reldual}
If $(X,0)$ is a \RHS surface singularity, then
\begin{enumerate}[leftmargin=3cm,itemindent=.5cm,labelwidth=\itemindent,labelsep=.5cm,align=left,label={$(\alph*)$}]
\item 
$\mathrm{mpc}^{\pi_{\I}(\cS'_{\mathbb{R}})}((f_{\ell'_0})_0(\tt_{\I}))=\qp{q}_{-h_0,\I}(-\check{\ell}_{0})=q_{[\ZK]+h_0,\I}(\ZK+\ell'_0);$
\item 
$\mathrm{pc}^{\pi_{\I}(\cS'_{\mathbb{R}})}((f_{\ell'_0})_0(\tt_{\I}))=\qp{Q}_{-h_0,\I}(-\check{\ell}_{0})=Q_{[\ZK]+h_0,\I}(\ZK+\ell'_0).$
\end{enumerate}
\end{cor}

\begin{proof}
Specialize Theorem~\ref{tdual} to $h=0$.
\end{proof}

\section{Application of the twisted duality to invariants of curves on rational surface singularities}
\label{sec:deltakappa}
\subsection{}
According to Blache~\cite{Blache-RiemannRoch}, for every algebraic normal surface germ $(X,0)$ there
exists a unique map $A_{X,0}:{\rm Weil}(X,0)/{\rm Cartier}(X,0) \to \QQ$ with the following properties
\begin{align}
\nonumber
&A_{X,0}(-D)=A_{X,0}(-K_X+D) && \mbox{ for any Weil divisor } \ D,\\
&A_{X,0}(C)= \chi(-\ell_C')-\delta(C) && \mbox{ for any reduced curve } \ (C,0)\subset (X,0). \label{eq:A2}
\end{align}
This map appears as the correction term of the generalized Riemann-Roch formula and adjunction formulae for Weil divisors on projective algebraic normal surfaces, and measures the local contribution of a singular point of the surface. For more details we refer to \cite{Blache-RiemannRoch} (or see the discussion from \cite[1.5]{kappa}).

\subsection{} We consider a reduced curve germ $(C,0)$ on a rational normal surface singularity $(X,0)$. Let $\pi:\tilde{X}\rightarrow X$
be a good embedded resolution of $C\subset X$ and denote by $\Gamma$ its dual resolution graph.
Although it is not necessary, for the transparency of the proof we will also assume
condition~\eqref{rescond} for the resolution, that is, the strict transform $\widetilde{C}$ meets each $E_v$ ($v\in \V$) in at
most one point. As before, we denote by $\ell'_C\subset L'$ the exceptional part of $\pi^*C$, that is,
$\pi^*C=\widetilde C+\ell'_C$ and $I_C=\Supp^*(\ell'_C)$, the subset of
irreducible exceptional divisors which intersect $\widetilde{C}$.

The aim of this section is to apply the general delta invariant formula (Theorem~\ref{thm:delta-sec}) 
and the twisted duality results of the previous section~\ref{ss:td} to give a topological/combinatorial 
proof for the following results, already presented in~\cite{kappa}.
\begin{thm}[\cite{kappa}]\label{thm:deltakappa}
Let $(X,0)$ be a rational surface singularity and $C\subset X$ a reduced curve germ, then
\begin{enumerate}[leftmargin=3cm,itemindent=.5cm,labelwidth=\itemindent,labelsep=.5cm,align=left,label={$(\arabic*)$}]
 \item
\label{eqfinal:delta1} 
$\delta(C)=
\kappa_{\Gamma,I_C}(\ell'_C)
=\kappa_X(C)=\chi(Z_K+\ell'_C)-\chi(s_{[Z_K+\ell'_C]});$
\item
\label{eqfinal:delta2} 
$A_{X,0}(C)=\chi(s_{[Z_K+\ell'_C]})=\chi(s_{[-\ell'_C]}).$
\end{enumerate}
\end{thm}

\begin{rem}
We emphasize that the first identity from Theorem~\ref{thm:deltakappa} has the same spirit as the results
presented in section~\ref{sec:pg}: it shows how to decode important invariants (in this case $\delta(C)$)
from the Poincar\'e series, since the relative $\kappa$-invariant $\kappa_{\Gamma,I_C}(\ell'_C)$ is defined
as its counting function. In forthcoming papers we plan to understand this concept
in full generality, i.e.~in the case of reduced curves embedded in any normal surface singularity.
\end{rem}

\begin{proof}
The expression~\ref{eqfinal:delta1} follows from 
the assumption on the resolution, we can write the cycle associated with the curve $C=\cup_{i\in I}C_i$ as
$\ell'_C=\sum_{i\in I}E^*_i$ where $I=I_C\subset \V$. We also define the cycles $\ell'_{C_J}=\sum_{j\in J}E^*_j$
associated with the curve $C_J:=\cup_{j\in J}C_j$ for any $J\subset I$, and denote $h_J:=[\ell'_{C_J}]\in H$. 
For convenience, we set $\ell'_{C_{\emptyset}}=0$, so $h_{C_{\emptyset}}=0\in H$.

Consider the relative topological Poincar\'e series of $C_J$ which was defined in~\eqref{eq:rtPs} by the equation
\begin{equation}\label{eq:rZJ}
Z^{(C_J)}(\mathbf{t})=Z(\mathbf{t})\cdot \prod_{j\in J}(1-\mathbf{t}^{E^*_j}) \ \ \ \mbox{for any} \ J\subset I.
\end{equation}
Then, as it is explained in section~\ref{ss:rCDGZ}, one has the decomposition 
$Z^{(C_J)}(\mathbf{t})=\sum_{h\in H} Z_h^{(C_J)}(\mathbf{t})$, and the relative CDGZ-identity (Theorem~\ref{relCDGZ}) provides for $\emptyset\neq J\subset I$:
\begin{equation}\label{eq:CDGZ}
P_{C_{J}}(\mathbf{t}_{J})=Z_0^{(C_J)}(\mathbf{t}_{J}).
\end{equation}
For $J=\{j\}$, the formula $\pc(P_{C_{j}}(t_{j}))=\pc(Z_0^{(C_j)}(t_{j}))$ follows from~\eqref{eq:CDGZ}.
Otherwise, $|J|\geq 2$, which implies $P_{C_{J}}$ is a polynomial by Lemma~\ref{lem:Ppoly}, and thus
$P_{C_{J}}(1_{J})=\pc^{\pi_J(S'_\RR)}(Z_0^{(C_J)}(\tt_{J}))$.

Moreover, thanks to~\ref{eq:deltapc} one can express $\delta(C)$ as
\begin{equation}\label{eq:delta2}
\delta(C)=\sum_{\smallmatrix \emptyset\neq J\subset I \endsmallmatrix}
(-1)^{|J|} \pc^{\pi_J(S'_\RR)}(Z_0^{(C_J)}(\tt_{J})).
\end{equation}
From the definition~\eqref{eq:rZJ} of $Z^{(C_J)}$, its $0$-part can be expressed as
$$Z_0^{(C_J)}(\tt_{J})=\sum_{\K\subset J}(-1)^{|\K|}\tt_{J}^{\ell'_{C_\K}}\cdot Z_{-h_{\K}}(\tt_{J}).$$
Hence, $Z_0^{(C_J)}(\tt_{J})$ is the alternating sum of twisted zeta functions
$\tt_{J}^{\ell'_{C_\K}} Z_{-h_{\K}}(\tt_{J})$ whose counting functions are $Q_{-h_{\K},J}(\ell-\ell'_{C_\K})$
by section~\ref{ss:td}. Moreover, this implies that the following function on $L$:
\begin{equation}\label{eq:D}
D_C(\ell):=\sum_{\smallmatrix \emptyset\neq J\subset \I \endsmallmatrix}(-1)^{|J|}
\sum_{\K\subset J}(-1)^{|\K|} Q_{-h_{\K},J}(\ell-\ell'_{C_\K})
\end{equation}
admits a quasi-polynomial associated with the cone $\pi_J(S'_\RR)$, denoted by
$\widetilde{D}_C(\ell)$, and hence~\eqref{eq:delta2} implies $\widetilde{D}_C(0)=\delta(C)$.

On the other hand, we can rearrange the terms in the above definition of $D_C(\ell)$ so that
$$D_C(\ell)=Q_{-h_C,\I}(\ell-\ell'_{C})+\sum_{\K\subsetneq \I} (-1)^{|\K|}R_{-h_{\K}}(\ell-\ell'_{C_\K}),$$
where $R_{-h_{\K}}(\ell-\ell'_{C_\K}):=
\sum_{\smallmatrix J\neq\emptyset\\ \K\subset J\subset \I \endsmallmatrix}
(-1)^{|J|}Q_{-h_{\K},J}(\ell-\ell'_{C_\K})$, since the sign of the first term associated with $\K=J=\I$ is $(-1)^{2|\I|}$.
Therefore, using the dual shift $\check{\ell}_C$ given by~\eqref{eq:dualrep}, on the quasi-polynomial level one gets
\begin{equation*}
\widetilde{D}_C(\ell)=\qp{Q}_{-h_C,\I}(\ell-\check{\ell}_C)+
\sum_{\K\subsetneq \I} (-1)^{|\K|}\qp{R}_{-h_{\K}}(\ell-\check{\ell}_{C_\K}),
\end{equation*}
with the notation
\begin{equation}\label{eq:frakR}
\qp{R}_{-h_{\K}}(\ell-\check{\ell}_{C_\K}):=
\sum_{\smallmatrix J\neq\emptyset \\ \K\subset J\subset \I \endsmallmatrix}
(-1)^{|J|}\qp{Q}_{-h_{\K},J}(\ell-\check{\ell}_{C_\K}).
\end{equation}
Then, after substituting $\ell=0$ above, one can apply the relative duality Corollary~\ref{reldual}
to $\qp{Q}_{-h_C,\I}(-\check{\ell}_C)$ which provides the final equation
\begin{equation}
\delta(C)=
Q_{[\ZK+\ell'_C],I}(\ZK+\ell'_C)+
\sum_{\K\subsetneq \I} (-1)^{|\K|}\qp{R}_{-h_{\K}}(-\check{\ell}_{C_\K}).
\end{equation}
By Theorem~\ref{thm:eqratsing} we have 
$Q_{[\ZK+\ell'_C],I}(\ZK+\ell'_C)=:\kappa_{\Gamma,I}(\ell'_C)=\kappa_{X,I}(\ell'_C)=\kappa_X(C)$.
Hence, to finish the proof it remains to show the vanishing of the sum in the final equation.
\vspace{0.3cm}

More precisely, we will prove that $\qp{R}_{-h_{\K}}(-\check{\ell}_{C_\K})=0$ for any $\K\subsetneq I$.
We apply the surgery formula to express $Q_{-h_{\K},J}(\ell-\ell'_{C_\K})$.

Thus, for the fixed subsets $\K\subset J\subset \I\subset \V$ ($J\neq\emptyset$) if we denote by
$\{\Gamma_k\}_k$ the connected full subgraphs determined by the subset of vertices $V\setminus J$,
the surgery formula~\eqref{eq:surgform} gives for a \emph{sufficiently large} $\ell$ the expression
$$Q_{-h_{\K},J}(\ell-\ell'_{C_\K})=
Q_{-h_{\K}}(\ell-\ell'_{C_\K})-\sum_k Q^{\Gamma_k}_{0}(\ell|_{\Gamma_k}),$$
since $j^*_k(\ell-\ell'_{C_\K})=j^*_k(\ell)=:\ell|_{\Gamma_k}$ whenever $\K\subset J$.
This, on the quasi-polynomial level with the substitution $\ell=0$ becomes
$$\qp{Q}_{-h_{\K},J}(-\check{\ell}_{C_\K})=
\qp{Q}_{-h_{\K}}(-\check{\ell}_{C_\K})-\sum_k \qp{Q}^{\Gamma_k}_{0}(0),$$
which implies that
$$\qp{Q}_{-h_{\K},J}(-\check{\ell}_{C_\K})=
\qp{Q}_{-h_{\K}}(-\check{\ell}_{C_\K})$$
since the periodic constant
$\qp{Q}^{\Gamma_k}_{0}(0)$ associated with the subgraph $\Gamma_k$ of the dual graph of a
rational singularity is $0$ (see the last line of proof of Corollary~\ref{cor:surgform}).

If $K=\emptyset$, then $\check\ell_{C_K}=0$ and hence $\qp{Q}_0(0)=0$ as above. This shows in this case 
$\qp{R}_{-h_{\K}}(-\check{\ell}_{C_\K})=0$.

Otherwise, $K\neq \emptyset$, then $\check\ell_{C_K}=\ell_{C_K}+E$ and thus one obtains
$$\qp{R}_{-h_{\K}}(-\check{\ell}_{C_\K})=\qp{R}_{-h_{\K}}(-\check{\ell}_{C_\K})=
\sum_{\smallmatrix \K\subset J\subset \I\endsmallmatrix}
(-1)^{|J|}\qp{Q}_{-h_{\K}}(-\check{\ell}_{C_\K})=0,$$
since $\qp{Q}_{-h_{\K}}(-\check{\ell}_{C_\K})$ does not depend of $J$ and
$$\sum_{\smallmatrix \K\subset J\subset \I\endsmallmatrix} (-1)^{|J|}=
(-1)^{|K|}\sum_{k=0}^{|I|-|K|}{{|I|-|K|}\choose{k}}(-1)^k = (-1)^{|K|}(1-1)^{|I|-|K|}=0.$$
On the other hand, \eqref{eq:SUMQP} combined with Corollary~\ref{cor:sharp}\ref{cor:sharp1} give
\begin{equation}\label{eq:Qform}
\widetilde{Q}_{[\ell']}(\ell)=\chi(\ell')-\chi(s_{[\ell']}) \ \ \mbox{ where } \ \ell'=\ell+r_{[\ell']}.
\end{equation}
This, together with Corollary~\ref{cor:sharp}\ref{cor:sharp2} implies that
$$Q_{[Z_K+\ell'_C]}(Z_K+\ell'_C)=\widetilde{Q}_{[Z_K+\ell'_C]}(Z_K+\ell'_C)=
\chi(Z_K+\ell'_C)-\chi(s_{[Z_K+\ell'_C]}),$$
which completes the proof of part~\ref{eqfinal:delta1}.

\vspace{0.3cm}
As for part~\ref{eqfinal:delta2}, the first identity is immediate by~\eqref{eq:A2} and part~\ref{eqfinal:delta1}. 
In order to show the identity $\chi(s_{[Z_K+\ell'_C]})=\chi(s_{[-\ell'_C]})$ (not true for general normal surface 
singularities, cf. \cite[Example~4.5]{kappa}) can be shown as follows.

The twisted duality Theorem~\ref{tdual}\ref{tdual:2} with $h=0$ and $\ell'_0=\ell_C'$ 
(and $\check{\ell}_C'=\ell_C'-r_{[-\ell'_C]}$, cf. \ref{sec:tw-series}) implies
$Q_{[Z_K+\ell'_C]}(Z_K+\ell'_C)=\widetilde{Q}_{[-\ell_C']}(-\check{\ell}_{C}')$. Moreover, by applying~\eqref{eq:Qform} for
$\ell':= -\ell_C'=-\check{\ell}'_C+r_{[-\ell_C']}$ we obtain
$\widetilde{Q}_{[-\ell_C']}(-\check{\ell}_{C}')=\chi(-\ell_C')-\chi(s_{[-\ell_C']})$, which deduces the desired identity
since $\chi(Z_K+\ell'_C)=\chi(-\ell_C')$.
\end{proof}

\section{Examples}
\label{sec:examples}

\subsection{A cyclic quotient}
 Consider $X$ the cyclic quotient singularity $\frac{1}{4}(1,3)$, whose resolution graph is the
 $\mathbb A_3$ graph. The action is $\xi*(x,y)=(\xi x, \xi^3 y)$ ($\xi^4=1$), hence the invariant ring is
 generated by
 $u=x^4$, $v=y^4$ and $w=xy$. In particular, $X=\{uv=w^4\}$.
 \begin{center}
 \begin{figure}[ht]
 \begin{tikzpicture}
\node[shape=circle,fill, draw=black,scale=.7] (A) at (-3,0) {} node [below=4] at (A) {$E_1 (-2)$}
node [above=10] at (A) {$E_1^*=(\frac{3}{4},\frac{1}{2},\frac{1}{4})$};
\node[shape=circle,fill, draw=black,scale=.7] (B) at (0,0) {} node [below=4] at (B) {$E_2 (-2)$}
node [above=10] at (B) {$E_2^*=(\frac{1}{2},1,\frac{1}{2})$};
\node[shape=circle,fill, draw=black,scale=.7] (C) at (3,0) {} node [below=4] at (C) {$E_3 (-2)$}
node [above=10] at (C) {$E_3^*=(\frac{1}{4},\frac{1}{2},\frac{3}{4})$};
\path [-,very thick] (A) edge (B);
\path [-,very thick] (B) edge (C);
 \end{tikzpicture}
 \caption{Dual graph of cyclic quotient singularity}
 \end{figure}
 \end{center}
The coordinates of $E^*_i$ are in terms of the basis $\{E_1,E_2,E_3\}$.
If $f(x,y)=x^{12}-y^{4}$, then $f$ is invariant, hence the corresponding divisor
$C=c(\{f=0\})$ is Cartier (where $c:\CC^2\to X$ is the universal abelian covering, cf. section \ref{ss:uac}).
It is given by $u^3=v$ on $X$. Therefore, $C$ is $\{uv-w^4=v-u^3=0\}$,
isomorphic to the plane curve singularity $\{w^4=u^4\}$ with $\delta(C)=6$.
Also note that $\ell'_C=4E^*_1=(3,2,1)$, hence $r_{h_C}=s_{h_C}=0$. The topological Poincar\'e series is given by
 $$
 \array{rcl}
 Z(\tt)&=&
 T\Big(\frac{1}{(1-t_1^{\frac{3}{4}}
 t_2^{\frac{1}{2}}t_3^{\frac{1}{4}})(1-t_1^{\frac{1}{4}}t_2^{\frac{1}{2}}t_3^{\frac{3}{4}})}\Big)
 =\displaystyle{\sum_{h\in \ZZ_4}} Z_h(\tt)\\ \\
 &=&\Big(1+t_1t_2t_3+t_1t_2^2t_3^3+t_1^2t_2^2t_3^2+t_1^2t_2^3t_3^4+t_1^2t_2^4t_3^6+
 \displaystyle{\sum_{\ell'\in L,\ \ell'\geq \ell'_C}\ z(\ell')\tt^{\ell'}}\Big)
 +\displaystyle{\sum_{0\neq h\in \ZZ_4}} Z_h(\tt),
 \endarray
 $$ where $T(\cdot)$ is the Taylor expansion at the origin of the given rational function.
 This together with $\ZK=0$ implies that
 $\kappa_{X}(C)=Q_h(\ZK+\ell'_C)=Q_h(\ell'_C)=6$.

 Analogously, $f(x,y)=x^2-y^2\in (\cO_Y)_{[2]}$ defines a Weil divisor with $h=[2]\in \ZZ_4$.
 The two components of $f=0$ are sent by $c$ into an irreducible $C$,
 with equation $\{u=w^2=v\}$. Hence $C$ is smooth with $\delta(C)=0$.
 In this case $\ell'_C=E^*_2=(\frac{1}{2},1,\frac{1}{2})$, and, in fact,
 $s_{[2]}=\ell'_C$
 (though $r_2=(\frac{1}{2},0,\frac{1}{2})$). Therefore,
 $\kappa_X(C)=Q_h(\ell'_C)=Q_h(s_{[2]})=0$ by Corollary \ref{cor:sharp}.

\subsection{A non-cyclic quotient singularity}
\label{ex:dihedral}
 Consider the binary dihedral quotient singularity $X=\CC^2/G$ where
 $$G=\mathbb{B}\mathbb{D}_{12}(7)=\left\langle
 \alpha=\left(\array{cc} 0&-1\\ 1&0\endarray\right),
 \beta=\left(\array{cc} \xi &0\\ 0&\xi^7\endarray\right)\right\rangle,$$
 and $\xi^{12}=1$ is a primitive root of unity.
 Its resolution graph is

 \begin{center}
 \begin{figure}[ht]
\begin{tikzpicture}[scale=.5]
\node [shape=circle,fill,draw=black,scale=.7] (O) at (0,0) {}
node [left=4] at (O) {$(-3) E_2$} node [right=2] at (O) {$E^*_2=(\frac{1}{3}, \frac{2}{3}, \frac{1}{3}, \frac{1}{3})$};
\node[shape=circle,fill, draw=black,scale=.7] (A) at (-3,3) {} node [left=4] at (A) {$(-2) E_3$}
node [above=7] at (A) {$E^*_3=(\frac{1}{6}, \frac{1}{3}, \frac{2}{3}, \frac{1}{6})$};
\node[shape=circle,fill, draw=black,scale=.7] (B) at (3,3) {} node [right=4] at (B) {$E_4 (-2)$} node [above right=7] at (B) {$E^*_4=(\frac{1}{6}, \frac{1}{3}, \frac{1}{6}, \frac{2}{3})$};
\node[shape=circle,fill, draw=black,scale=.7] (C) at (0,-3) {} node [left=4] at (C) {$(-2) E_1$} node [below=7] at (C) {$E^*_1=(\frac{2}{3}, \frac{1}{3}, \frac{1}{6}, \frac{1}{6})$};
\path [-,very thick] (A) edge (0,0);
\path [-,very thick] (B) edge (0,0);
\path [-,very thick] (C) edge (0,0);
\end{tikzpicture}
\caption{Dual graph of dihedral quotient singularity}
 \end{figure}
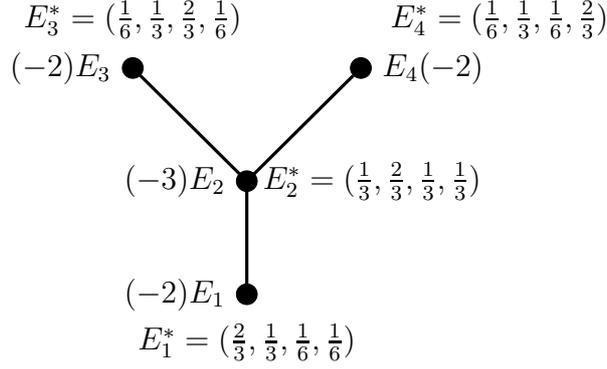
 \end{center}
 In this case we consider embedded curves $C_h$ with $\ell'_{C_h}=s_h$ for $h\in H$, $h\not=0$.
 For instance, if $h=[E^*_2]$ and $s_h=E^*_2$, then $\widetilde{C}_h$
 is a transversal cut (curvette) of $E_2$ (analogously, $2E^*_2$ means
 two disjoint transversal cuts of $E^*_2$, etc).
 $\ZK=E_2^*=\frac{1}{3}E_1+\frac{2}{3}E_2+\frac{1}{3}E_3+\frac{1}{3}E_4$.

 One shows that
$H=\ZZ_2[E^*_1-E^*_4]\times \ZZ_6[E^*_4]$, and $[3E^*_2]=[E_1^*+E^*_3+E^*_4]=0$ in $H$.
Hence by the cyclic symmetry of the graph we have only five different cases.
They correspond to
$h\in \{[E_2^*]=(0,2),[2E^*_2]=(0,4),[E_1^*]=(1,1),[E_1^*+E^*_2]=(1,3),[E_1^*+2E^*_2]=(1,5)\}$.

The following table shows the values $\kappa_X(C_h)=\delta(C_h)$ for different cases.
The column $Z_{[Z_K]+h}(\tt)$ only shows the
monomials $z(\ell')\tt^{\ell'}$ whose degree is $\ell'\ngeq Z_K+s_h$.

\vspace*{14pt}
 \begin{center}
 \begin{tabular}{|c|c|c|c|c|c|c|}
 $h\in H$ & $r_h$ & $s_h$ & $\kappa_X(C_h)$ & $Z_{[Z_K]+h}(\tt)$ & $\chi(-s_h)$ & $A_{X,0}(C_h)$\\ [5pt]
 \hline
 \raisebox{-.7ex}{$(0,2)$} & \raisebox{-.7ex}{$(\frac{1}{3}, \frac{2}{3}, \frac{1}{3}, \frac{1}{3})$} & \raisebox{-.7ex}{$(\frac{1}{3}, \frac{2}{3}, \frac{1}{3}, \frac{1}{3})$} & \raisebox{-.7ex}{$0$} & \raisebox{-.7ex}{$0+...$} & \raisebox{-.7ex}{$\frac{2}{3}$}& \raisebox{-.7ex}{$\frac{2}{3}$}\\[5pt]
 \hline
 \raisebox{-.7ex}{$(0,4)$} & \raisebox{-.7ex}{$(\frac{2}{3}, \frac{1}{3}, \frac{2}{3}, \frac{2}{3})$} & \raisebox{-.7ex}{$(\frac{2}{3}, \frac{4}{3}, \frac{2}{3}, \frac{2}{3})$} & \raisebox{-.7ex}{2} & \raisebox{-.7ex}{$1+\tt^{(1,1,1,1)}+...$} & \raisebox{-.7ex}{$2$}& \raisebox{-.7ex}{$0$}\\[5pt]
 \hline
 \raisebox{-.7ex}{$(1,1)$} & \raisebox{-.7ex}{$(\frac{2}{3}, \frac{1}{3}, \frac{1}{6}, \frac{1}{6})$} & \raisebox{-.7ex}{$(\frac{2}{3}, \frac{1}{3}, \frac{1}{6}, \frac{1}{6})$} & \raisebox{-.7ex}{$0$} & \raisebox{-.7ex}{$0+...$} & \raisebox{-.7ex}{$\frac{1}{2}$} & \raisebox{-.7ex}{$\frac{1}{2}$}\\[5pt]
 \hline
 \raisebox{-.7ex}{$(1,3)$} & \raisebox{-.7ex}{$(0, 0, \frac{1}{2}, \frac{1}{2})$} & \raisebox{-.7ex}{$(1, 1, \frac{1}{2}, \frac{1}{2})$} & \raisebox{-.7ex}{1} & \raisebox{-.7ex}{$\tt^{(\frac{1}{6},\frac{1}{3},\frac{2}{3},\frac{1}{6})}+...$} & \raisebox{-.7ex}{$\frac{3}{2}$} & \raisebox{-.7ex}{$\frac{1}{2}$}\\[5pt]
 \hline
 \raisebox{-.7ex}{$(1,5)$} & \raisebox{-.7ex}{$(\frac{1}{3},\frac{2}{3},\frac{5}{6},\frac{5}{6})$} & \raisebox{-.7ex}{$(\frac{1}{3}, \frac{2}{3}, \frac{5}{6}, \frac{5}{6})$} & \raisebox{-.7ex}{1} & \raisebox{-.7ex}{$\tt^{(\frac{2}{3},\frac{1}{3},\frac{1}{6},\frac{1}{6})}+...$} & \raisebox{-.7ex}{$\frac{7}{6}$} & \raisebox{-.7ex}{$\frac{1}{6}$}\\[5pt]
 \hline
 \end{tabular}
 \end{center}

 \vspace{2mm}

\subsection{A non-rational singularity}
\label{ex:non-rational}
Consider the singularity $X=(\{x^2+y^3+z^7\},0)$ whose dual good resolution graph $\Gamma$ is 

 \begin{center}
 \begin{figure}[ht]
\begin{tikzpicture}[scale=.5]
\node[] [shape=circle,fill, draw=black,scale=.7] (O) at (0,0) {} node [right=4] at (O) {$E_1 (-1)$};
\node[shape=circle,fill, draw=black,scale=.7] (A) at (-3,3) {} node [above=4] at (A) {$E_2 (-2)$};
\node[shape=circle,fill, draw=black,scale=.7] (B) at (3,3) {} node [above=4] at (B) {$E_3 (-3)$};
\node[shape=circle,fill, draw=black,scale=.7] (C) at (0,-3) {} node [below=4] at (C) {$E_4 (-7)$};
\path [-,very thick] (A) edge (0,0);
\path [-,very thick] (B) edge (0,0);
\path [-,very thick] (C) edge (0,0);
\end{tikzpicture}
\caption{Triangle $2,3,7$}
 \end{figure}
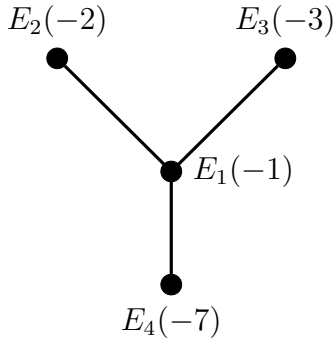
 \end{center}
with $\ZK=2E_1+E_2+E_3+E_4$ and $H=\{0\}$.
\begin{enumerate}
\item\label{ex:non-rational1}
Consider the germ given by $C=(z,x^2+y^3)$ for which one has $\ell'_C=E_4^*=(6,3,2,1)$.
Then one calculates $\kappa_{\Gamma}(\ell'_C)=Q_0(\ZK+\ell'_C)=2$ given by the coefficients:
$$Z(\tt)=\frac{(1-\tt^{E_1^*})}{(1-\tt^{E_2^*})(1-\tt^{E_3^*})(1-\tt^{E_4^*})}=
1+\tt^{(6,3,2,1)}+...\text{ terms }\tt^{\ell'}, \ell'\geq \ZK+\ell'_C.$$
Analogously, $\kappa_{\Gamma,I_C}(\ell'_C)=2$.
However, $\delta(C)=1$ since $C$ is the ordinary plane cusp.

Let $f\in \cO_X$ be a function and $\ell'_f$ denotes its corresponding cycle, ie. 
the exceptional part of $\pi^*(f)$. Then, for the coordinate functions one has 
$\ell'_x=(21,11,7,3)$, $\ell'_y=(14,7,5,2)$, and $\ell'_z=(6,3,2,1)$. 
Hence $\cF(\ZK+\ell'_C)_0=\langle x,y,z^2\rangle$,
which implies that $\kappa_X(C)=\kappa_{X,I_C}(\ell'_C)=2$. Thus, for $C$ we have
$$\kappa_X(C)=\kappa_{X,I_C}(\ell'_C)=\kappa_{\Gamma}(\ell'_C)=\kappa_{\Gamma,I_C}(\ell'_C)\neq \delta(C).$$

\item\label{ex:non-rational2}
On the other hand, if we consider $C_1=(y+z^2,x-z^3\sqrt{1-z})$, then $\ell'_{C_1}=E_4^*$, 
hence it is embedded topologically equivalent with $C$.
However, one shows that $C_1$ is smooth (see eg.~\cite{Nem-PS}), hence $\delta(C_1)=0$.

\item\label{ex:non-rational3}
Consider now a new good resolution $\tilde\pi$ resulting from $\pi$ by blowing-up the point of intersection
of the strict transform of $C$ and $E_4$. This point $P$ is the base point
of the maximal ideal by $\pi$, that is,
$\pi^*\mathfrak{m}_{X}\mathcal{O}_{X}=\mathfrak{m}_P\cO_{\tilde X}(-Z_{\min})$. Then $\tilde I_{C}=E_5$,
$(\tilde \ell'_C)_{E_5}=2$, $(\tilde \ell'_x)_{E_5}=3$, $(\tilde\ell'_y)_{E_5}=2$, $(\tilde\ell'_z)_{E_5}=2$,
and $(\tilde Z_K)_{E_5}=0$, where $E_5$ is the new exceptional divisor and $(\ell'_f)_{E_5}$ represents the
$E_5$-coordinate of $(\ell'_f)$. Hence $\cF((\tilde Z_K+\tilde\ell'_C)_{\tilde I_C})=\langle x,y,z\rangle$,
and thus $\kappa_{X,\tilde I_C}(\ell'_C)=1$. On the other hand, $\kappa_{X,\tilde I_{C_1}}(\tilde\ell'_{C_1})=2$
since $\tilde I_{C_1}=I_{C_1}=E_4$. One can also check that $\kappa_{\tilde\Gamma}(\tilde\ell'_C)=2$, whereas
$\kappa_{\tilde\Gamma,\tilde I_C}(\tilde\ell'_C)=1$.
\end{enumerate}

\bibliographystyle{amsplain}

\begin{thebibliography}{10}

\bibitem{Blache-RiemannRoch}
R.~Blache, \textit{Riemann-{R}och theorem for normal surfaces and applications},
 Abh. Math. Sem. Univ. Hamburg \textbf{65} (1995), 307--340.

\bibitem{cdg}
A.~Campillo, F.~Delgado, and S.~M. Gusein-Zade, \textit{The {A}lexander
 polynomial of a plane curve singularity via the ring of functions on it},
 Duke Math. J. \textbf{117} (2003), no.~1, 125--156.

\bibitem{CDG-Poincare}
\bysame, \textit{Poincar\'{e} series of a rational surface singularity}, Invent.
 Math. \textbf{155} (2004), no.~1, 41--53.

\bibitem{CDGZ-PScurves}
\bysame, \textit{Poincar\'{e} series of curves on rational surface
 singularities}, Comment. Math. Helv. \textbf{80} (2005), no.~1, 95--102.

\bibitem{CDGZ-equivariantPS}
\bysame, \textit{Equivariant {P}oincar\'{e} series of filtrations}, Rev. Mat.
 Complut. \textbf{26} (2013), no.~1, 241--251.

\bibitem{CDGZ-HilbertFunction}
\bysame, \textit{Hilbert function, generalized {P}oincar\'{e} series and topology
 of plane valuations}, Monatsh. Math. \textbf{174} (2014), no.~3, 403--412.

\bibitem{CDGZ-equivariantPSZeta}
\bysame, \textit{An equivariant {P}oincar\'{e} series of filtrations and
 monodromy zeta functions}, Rev. Mat. Complut. \textbf{28} (2015), no.~2,
 449--467.

\bibitem{CDGZ-equivariantPSTopology}
\bysame, \textit{Equivariant {P}oincar\'{e} series and topology of valuations},
 Doc. Math. \textbf{21} (2016), 271--286.

\bibitem{CDGZ-PSfiltrations}
A.~Campillo, F.~Delgado, and S.M. Gusein-Zade, \textit{On {P}oincar\'{e} series
 of filtrations}, Azerb. J. Math. \textbf{5} (2015), no.~2, 125--139.

\bibitem{CDGZ-topologicaltype}
\bysame, \textit{On the topological type of a set of plane valuations with
 symmetries}, Math. Nachr. \textbf{290} (2017), no.~13, 1925--1938.

\bibitem{CDK}
A.~Campillo, F.~Delgado, and K.~Kiyek, \textit{Gorenstein property and symmetry for one-dimensional local Cohen-Macaulay rings},
Manuscripta Math. \textbf{83} (1994), 405--423.


\bibitem{ji-tesis}
J.I. Cogolludo-Agust{\'{\i}}n, \textit{Topological invariants of the complement
 to arrangements of rational plane curves}, Mem. Amer. Math. Soc. \textbf{159}
 (2002), no.~756, xiv+75.

\bibitem{kappa}
J.I. Cogolludo-Agust\'{\i}n, T.~L\'{a}szl\'{o}, J.~Mart\'{\i}n-Morales, and
 A.~N\'{e}methi, \textit{Delta invariant of curves on rational surfaces {I}.
 {T}he analytic approach}, arXiv:1911.07539 [math.AG].

\bibitem{CM}
J.I. Cogolludo-Agust\'{\i}n and J.~Mart\'{\i}n-Morales, \textit{The correction
 term for the {R}iemann-{R}och formula of cyclic quotient singularities and
 associated invariants}, Rev. Mat. Complut. \textbf{32} (2019), no.~2,
 419--450.

\bibitem{ji-JJ-numerical}
J.I. Cogolludo-Agust{\'{\i}}n, J.~{Mart{\'i}n-Morales}, and
 J.~{Ortigas-Galindo}, \textit{Numerical adjunction formulas for weighted
 projective planes and lattice point counting}, Kyoto J. Math. \textbf{56}
 (2016), no.~3, 575--598.

\bibitem{GN}
E.~Gorsky and A.~N\'{e}methi, \textit{Lattice and {H}eegaard {F}loer homologies
 of algebraic links}, Int. Math. Res. Not. IMRN (2015), no.~23, 12737--12780.

\bibitem{CDGZ-Moncurve}
S.~M. Guse\u{\i}n-Zade, F.~Del{'}gado, and A.~Kampil{'}o, \textit{On the
 monodromy of a plane curve singularity and the {P}oincar\'{e} series of its
 ring of functions}, Funktsional. Anal. i Prilozhen. \textbf{33} (1999),
 no.~1, 66--68.

\bibitem{cdg2}
\bysame, \textit{Integrals with respect to the {E}uler characteristic over spaces
 of functions, and the {A}lexander polynomial}, Tr. Mat. Inst. Steklova
 \textbf{238} (2002), no.~Monodromiya v Zadachakh Algebr. Geom. i Differ.
 Uravn., 144--157.

\bibitem{CDGZ-PSuniversalAC}
\bysame, \textit{Universal abelian covers of rational surface singularities, and
 multi-index filtrations}, Funktsional. Anal. i Prilozhen. \textbf{42} (2008),
 no.~2, 3--10, 95.

\bibitem{Hironaka}
H.~Hironaka, \textit{On the arithmetic genera and the effective genera of
 algebraic curves}, Mem. Coll. Sci. Univ. Kyoto Ser. A. Math. \textbf{30}
 (1957), 177--195.

\bibitem{LNN}
T.~L\'aszl\'o, J.~Nagy and A.~N\'emethi, \textit{Surgery formulae for the
 {S}eiberg--{W}itten invariant of plumbed 3-manifold}, Rev Mat Complut (2020) \textbf{33}:197-230.

\bibitem{LNNdual}
\bysame, \textit{Combinatorial duality for {P}oincar\'e series, polytopes and
 invariants of plumbed 3-manifolds}, Selecta Math
(2019) 25:21, https://doi.org/10.1007/s00029-019-0468-9.

\bibitem{LNehrhart}
T.~L\'{a}szl\'{o} and A.~N\'{e}methi, \textit{Ehrhart theory of polytopes and
 {S}eiberg-{W}itten invariants of plumbed 3-manifolds}, Geom. Topol.
 \textbf{18} (2014), no.~2, 717--778. \MR{3180484}

\bibitem{LSzPoincare}
T.~L\'{a}szl\'{o} and Z.~Szil\'{a}gyi, \textit{On {P}oincar\'{e} series
 associated with links of normal surface singularities}, Trans. Amer. Math.
 Soc. \textbf{372} (2019), no.~9, 6403--6436. \MR{4024526}

\bibitem{LSzMonoids}
T.~L\'{a}szl\'{o} and Zs. Szil\'{a}gyi, \textit{Non-normal affine monoids,
 modules and {P}oincar\'{e} series of plumbed 3-manifolds}, Acta Math. Hungar.
 \textbf{152} (2017), no.~2, 421--452. \MR{3682892}


\bibitem{Libgober-alexander}
A.~Libgober, \textit{Alexander polynomial of plane algebraic curves and cyclic
 multiple planes}, Duke Math. J. \textbf{49} (1982), no.~4, 833--851.

\bibitem{Libgober-characteristic}
\bysame, \textit{Characteristic varieties of algebraic curves}, Applications of
 algebraic geometry to coding theory, physics and computation (Eilat, 2001),
 Kluwer Acad. Publ., Dordrecht, 2001, pp.~215--254.

\bibitem{Lipman} J. Lipman, \textit{ Rational singularities, with applications to algebraic surfaces
and unique factorization}, Inst. Hautes \'Etudes Sci. Publ. Math. {\textbf 36} (1969), 195--279.

\bibitem{Moy}
J.~J. Moyano-Fern\'{a}ndez, \textit{Poincar\'e series for curve singularities and its behaviour under
projections}, J. of Pure and Appl. Alg. \textbf{219} (2015), 2449--2462.

\bibitem{MZG}
J.~J. Moyano-Fern\'{a}ndez and W.~A. Z\'{u}\~{n}iga Galindo, \textit{Motivic zeta
 functions for curve singularities}, Nagoya Math. J. \textbf{198} (2010),
 47--75.

\bibitem{NemOSZ}
A.~N\'{e}methi, \textit{On the {O}zsv\'{a}th-{S}zab\'{o} invariant of negative
 definite plumbed 3-manifolds}, Geom. Topol. \textbf{9} (2005), 991--1042.

\bibitem{Nem-PS}
\bysame, \textit{Poincar\'{e} series associated with surface singularities},
 Singularities {I}, Contemp. Math., vol. 474, Amer. Math. Soc., Providence,
 RI, 2008, pp.~271--297.

\bibitem{NJEMS}
\bysame, \textit{The {S}eiberg-{W}itten invariants of negative definite plumbed
 3-manifolds}, J. Eur. Math. Soc. (JEMS) \textbf{13} (2011), no.~4, 959--974.

\bibitem{Nem-CLB}
\bysame, \textit{The cohomology of line bundles of splice-quotient
 singularities}, Adv. Math. \textbf{229} (2012), no.~4, 2503--2524.

\bibitem{Nem-ICM}
\bysame, \textit{Pairs of invariants of surface singularities}, Proc. Int. Cong.
 of Math. 2018 Rio de Janeiro, vol.~1, 2018, pp.~745--776.

\bibitem{NN-SWI}
A.~N\'{e}methi and L.I. Nicolaescu, \textit{Seiberg-{W}itten invariants and
 surface singularities}, Geom. Topol. \textbf{6} (2002), 269--328.

\bibitem{NOk}
A.~N\'{e}methi and T.~Okuma, \textit{On the {C}asson invariant conjecture of
 {N}eumann-{W}ahl}, J. Algebraic Geom. \textbf{18} (2009), no.~1, 135--149.

\bibitem{Okuma-abelian}
T.~Okuma, \textit{Universal abelian covers of rational surface singularities}, J.
 London Math. Soc. (2) \textbf{70} (2004), no.~2, 307--324.

\bibitem{Ok}
\bysame, \textit{The geometric genus of splice-quotient singularities}, Trans.
 Amer. Math. Soc. \textbf{360} (2008), no.~12, 6643--6659.

\bibitem{Sakai84}
F.~Sakai, \textit{Weil divisors on normal surfaces}, Duke Math. J. \textbf{51}
 (1984), no.~4, 877--887.

\bibitem{torres}
G.~Torres, \textit{On the {A}lexander polynomial}, Ann. of Math. (2) \textbf{57}
 (1953), 57--89.

\end{thebibliography}
\providecommand{\bysame}{\leavevmode\hbox to3em{\hrulefill}\thinspace}
\providecommand{\MR}{\relax\ifhmode\unskip\space\fi MR }
\providecommand{\MRhref}[2]{%
 \href{http://www.ams.org/mathscinet-getitem?mr=#1}{#2}
}
\providecommand{\href}[2]{#2}

\end{document}